\documentclass[english]{amsart}
\usepackage[T1]{fontenc}
\usepackage[latin9]{inputenc}
\usepackage{amsthm}
\usepackage{amsmath}
\usepackage{amssymb}
\usepackage{esint}

\makeatletter
\theoremstyle{plain}
\newtheorem{thm}{\protect\theoremname}
  \theoremstyle{definition}
  \newtheorem{defn}[thm]{\protect\definitionname}
  \theoremstyle{plain}
  \newtheorem{lem}[thm]{\protect\lemmaname}
  \theoremstyle{remark}
  \newtheorem{rem}[thm]{\protect\remarkname}
  \theoremstyle{plain}
  \newtheorem{prop}[thm]{\protect\propositionname}

\DeclareMathOperator{\Id}{\text{Id}}

\makeatother

\usepackage{babel}
  \providecommand{\definitionname}{Definition}
  \providecommand{\lemmaname}{Lemma}
  \providecommand{\propositionname}{Proposition}
  \providecommand{\remarkname}{Remark}
\providecommand{\theoremname}{Theorem}

\begin{document}

\title[Construction of solutions]{Construction of solutions for a nonlinear elliptic
problem on Riemannian manifolds with boundary}

\thanks{The authors were  supported by Gruppo Nazionale per l'Analisi Matematica, la Probabilit\`{a} e le loro 
Applicazioni (GNAMPA) of Istituto Nazionale di Alta Matematica (INdAM)}

\author{Marco Ghimenti}
\address{M. Ghimenti, \newline Dipartimento di Matematica Universit\`a di Pisa
Largo B. Pontecorvo 5, 56126 Pisa, Italy}
\email{ghimenti@mail.dm.unipi.it}

\author{Anna Maria Micheletti}
\address{A. M. Micheletti, \newline Dipartimento di Matematica Universit\`a di Pisa
Largo B. Pontecorvo 5, 56126 Pisa, Italy}
\email{a.micheletti@dma.unipi.it.}

\keywords{Riemannian manifold with boundary, Nonlinear elliptic equations, Neumann boundary condition, Mean curvature, Liapounov Schmidt}

\subjclass[2010]{58J05,35J60,58E05}

\begin{abstract}
Let $(M,g)$ be a smooth compact, $n$ dimensional Riemannian manifold,
$n\ge2$ with smooth $n-1$ dimensional boundary $\partial M$. We
prove that the stable critical points of the mean curvature of the
boundary generates $H^{1}(M)$ solutions for the singularly perturbed
elliptic problem with Neumann boundary conditions 
\[
\left\{ \begin{array}{cc}
-\varepsilon^{2}\Delta_{g}u+u=u^{p-1} & \text{in }M\\
u>0 & \text{in }M\\
{\displaystyle \frac{\partial u}{\partial\nu}=0} & \text{on }\partial M
\end{array}\right.
\]
when $\varepsilon$ is small enough. Here $p$ is subcritical.
\end{abstract}
\maketitle
\section{Introduction}

Let $(M,g)$ be a smooth compact, $n$ dimensional Riemannian manifold,
$n\ge2$ with boundary $\partial M$ which is the union of a finite
number of connected, smooth, boundaryless, $n-1$ submanifolds embedded
in $M$. Here $g$ denotes the Riemannian metric tensor. By Nash theorem
\cite{Na54} we can consider $(M,g)$ as a regular submanifold embedded
in $\mathbb{R}^{N}$. 

We consider the following Neumann problem
\begin{equation}
\left\{ \begin{array}{cc}
-\varepsilon^{2}\Delta_{g}u+u=u^{p-1} & \text{in }M\\
u>0 & \text{in }M\\
{\displaystyle \frac{\partial u}{\partial\nu}=0} & \text{on }\partial M
\end{array}\right.\label{eq:P}
\end{equation}
where $p>2$ if $n=2$ and $2<p<2^{*}=\frac{2n}{n-2}$ if $n\ge3$,
$\nu$ is the external normal to $\partial M$ and $\varepsilon$
is a positive parameter. 

We are interested in finding solutions $u\in H^{1}(M)$ to problem
(\ref{eq:P}), where 
\[
{\displaystyle H^{1}(M)=\left\{ u:M\rightarrow\mathbb{R}:\int_{M}|\nabla u|_g^{2}+u^{2}d\mu_{g}<\infty\right\} }
\]
and $\mu_{g}$ denotes the volume form on $M$ associated to $g$.
More precisely, we want to show that, for $\varepsilon$ sufficiently
small, we can construct a solution which has a peak near a stable
critical point of the scalar curvature of the boundary, as stated
in the following.
\begin{defn}
Let $f\in C^{1}(N,\mathbb{R})$, where $(N,g)$ is a Riemannian manifold.
We say that $K\subset N$ is a $C^{1}$-stable critical set of $f$
if $K\subset\left\{ x\in N\ :\ \nabla f(x)=0\right\} $ and for
any $\mu>0$ there exists $\delta>0$ such that, if $h\in C^{1}(N,\mathbb{R})$
with 
\[
\max_{d_{g}(x,K)\le\mu}|f(x)-h(x)|+|\nabla f(x)-\nabla h(x)|\le\delta,
\]
then $h$ has a critical point $x_{0}$ with $d_{g}(x_{0},K)\le\mu$.
Here $d_{g}$ denotes the geodesic distance associated to the Riemannian
metric $g$.

Now we can state the main theorem.\end{defn}
\begin{thm}
\label{thm:main}Assume $K\subset\partial M$ is a $C^{1}$-stable
critical set of the mean curvature of the boundary. Then there exists
$\varepsilon_{0}>0$ such that, for any $\varepsilon\in(0,\varepsilon_{0})$,
Problem (\ref{eq:P}) has a solution $u_{\varepsilon}\in H^{1}(M)$
which concentrates at a point $\xi_{0}\in K$ as $\varepsilon$ goes
to zero.
\end{thm}
Problem (\ref{eq:P}) in a flat domain has a long history. Starting
from a problem of pattern formation in biology, Lin, Ni and Takagi
\cite{LNT,NT1} showed the existence of a mountain pass solution for
Problem (\ref{eq:P}) and proved that this solution has exactly one
maximum point which lies on the boundary of the domain. Moreover in
\cite{NT2} the authors proved that the maximum point of the solution
approaches the maximum point for the mean curvature of the boundary
when the perturbation parameter $\varepsilon$ goes to zero. 

Thenceforth, many papers were devoted to the study of Problem (\ref{eq:P})
on flat domains. In particular, in \cite{DFW,W1} it is proved that
any stable critical point of the mean curvature of the boundary generated
a single peaked solution whose peak approaches the critical point
as $\varepsilon$ vanishes. Moreover in \cite{G,Li,WW,GWW} the existence
of multipeak solutions whose peaks lies on the boundary is studied.
We also mention a series of works in which the authors proved the
existence of solutions which have internal peaks \cite{W98,GuW,GrPi00,GrPiW}.

In the case of a manifold $M$, Problem (\ref{eq:P}) has been firstly
studied in \cite{BP05} where the authors prove the existence of a
mountain pass solution when the manifold $M$ is closed and when the
manifold $M$ has a boundary. They show that for $\varepsilon$ small
such a solution has a spike which approaches -as $\varepsilon$ goes
to zero- a maximum point of the scalar curvature when $M$ is closed
and a maximum point of the mean curvature of the boundary when $M$
has a boundary. 

In the case of $M$ closed manifold, in \cite{BBM07} the authors
show that Problem (\ref{eq:P}) has at least cat$M+1$ non trivial
positive solutions when $\varepsilon$ goes to zero. Here cat$M$
denotes the Lusternik-Schnirelmann category of $M$. Moreover in \cite{MP09}
the effect of the geometry of the manifold $(M,g)$ is examined. In
fact it is shown that positive solution of the problem are generated
by stable critical points of the scalar curvature of $M$. 

More recently we proved in \cite{GMtmna} in the case of a manifold
$M$ with boundary $\partial M$ that Problem (\ref{eq:P}) has at
least cat$\partial M$ non trivial positive solutions when $\varepsilon$
goes to zero. We can compare the result of \cite{GMtmna} with Theorem
\ref{thm:main}. In fact, in \cite{GMjfpta} the authors prove that
generically with respect to the metric $g,$ the mean curvature of
the boundary has nondegenerate critical points. More precisely, the
set of metrics for which the mean curvature has only nondegenerate
critical points is an open dense set among all the $C^{k}$ metrics
on $M$, $k\ge3$. Thus, generically with respect to the metric, the
mean curvature has $P_{1}(\partial M)$ nondegenerate (hence stable)
critical points, where $P_{1}(\partial M)$ is the Poincar\'e polynomial
of $\partial M$, namely $P_{t}(\partial M)$, evaluated in $t=1$.
So, generically with respect to metric, Problem (\ref{eq:P}) has
$P_{1}(\partial M)$ solution, and it holds $P_{1}(\partial M)\ge\text{cat}\partial M$,
and in many cases the strict inequality holds.

The paper is organized as follows. In Section \ref{prel} some preliminary notions are introduced, which are necessary 
to the comprehension of the paper. In Section \ref{sec:Reduction} we study the variational structure of the problem and we 
perform the finite dimensional reduction. In Section \ref{proof} the proof of Theorem \ref{thm:main} is sketched while the 
expansion of the reduced functional is postponed in Section \ref{expansion}. The Appendix collects some technical lemmas.

\section{Preliminary results}\label{prel}

In this section we give some general facts preliminary to our work.
These results are widely present in literature, anyway, we refer
mainly to \cite{BP05,Es92,Es94,MP09} and the reference therein.

First of all we need to define a suitable coordinate chart on the
boundary. 

We know that on the tangent bundle of any compact Riemannian manifold
$\mathcal{M}$ it is defined the exponential map $\exp:T\mathcal{M}\rightarrow\mathcal{M}$
which is of class $C^{\infty}$. Moreover there exists a constant
$R_{M}>0$, called radius of injectivity, and a finite number of $x_{i}\in\mathcal{M}$
such that $\mathcal{M}=\cup_{i=1}^{l}B_{g}(x_{i},R_{M})$ and $\exp_{x_{i}}:B(0,R_{M})\rightarrow B_{g}(x_{i},R_{M})$
is a diffeormophism for all $i$. By choosing an orthogonal coordinate
system $(y_{1},\dots,y_{n})$ of $\mathbb{R}^{n}$ and identifying
$T_{x_{0}}\mathcal{M}$ with $\mathbb{R}^{n}$ for $x_{0}\in\mathcal{M}$
we can define by the exponential map the so called normal coordinates.
For $x_{0}\in\mathcal{M},$ $g_{x_{0}}$ denotes the metric read through
the normal coordinates. In particular, we have $g_{x_{0}}(0)=\Id$.
We set $\left\vert g_{x_{0}}(y)\right\vert =\det\left(g_{x_{0}}(y)\right)_{ij}$
and $g_{x_{0}}^{ij}(y)=\left(\left(g_{x_{0}}(y)\right)_{ij}\right)^{-1}$.
\begin{defn}
If $q$ belongs to the boundary $\partial M$, let $\bar{y}=\left(y_{1},\dots,y_{n-1}\right)$
be Riemannian normal coordinates on the $n-1$ manifold $\partial M$
at the point $q$. For a point $\xi\in M$ close to $q$, there exists
a unique $\bar{\xi}\in\partial M$ such that $d_{g}(\xi,\partial M)=d_{g}(\xi,\bar{\xi})$.
We set $\bar{y}(\xi)\in\mathbb{R}^{n-1}$ the normal coordinates for
$\bar{\xi}$ and $y_{n}(\xi)=d_{g}(\xi,\partial M)$. Then we define
a chart $\psi_{q}^{\partial}:\mathbb{R}_{+}^{n}\rightarrow M$ such
that $\left(\bar{y}(\xi),y_{n}(\xi)\right)=\left(\psi_{q}^{\partial}\right)^{-1}(\xi)$.
These coordinates are called \emph{Fermi coordinates} at $q\in\partial M$.
The Riemannian metric $g_{q}\left(\bar{y},y_{n}\right)$ read through
the Fermi coordinates satisfies $g_{q}(0)=\Id$. 
\end{defn}
We note by $d_{g}^{\partial}$ and $\exp^{\partial}$ respectively
the geodesic distance and the exponential map on by $\partial M$.
By compactness of $\partial M$, there is an $R^{\partial}$ and a
finite number of points $q_{i}\in\partial M$, $i=1,\dots,k$ such
that 
\[
I_{q_{i}}(R^{\partial},R_{M}):=\left\{ x\in M,\, d_{g}(x,\partial M)=d_{g}(x,\bar{\xi})<R_{M},\, d_{g}^{\partial}(q_{i},\bar{\xi})<R^{\partial}\right\} 
\]
form a covering of $\left(\partial M\right)_{R^M}:=\{ x\in M,\, d_{g}(x,\partial M)<R^M \}$ and on every
$I_{q_{i}}$ the Fermi coordinates are well defined. In the following
we choose, $R=\min\left\{ R^{\partial},R_{M}\right\} $, such that
we have a finite covering 
\[
M\subset\left\{ \cup_{i=1}^{k}B(q_{i},R)\right\} \bigcup\left\{ \cup_{i=k+1}^{l}I_{\xi_{i}}(R,R)\right\} 
\]
where $k,l\in\mathbb{N}$, $q_{i}\in M\smallsetminus\partial M$ and
$\xi_{i}\in\partial M$.

For $p\in\partial M$, consider $\pi_{p}:T_{p}M\rightarrow T_{p}\partial M$
the projection on the tangent space $T_{p}\partial M$. For a pair
of tangent vectors $X,Y\in T_{p}\partial M$ we define the second
fundamental form $II_{p}(X,Y):=\nabla_{X}Y-\pi_{p}(\nabla_{X}Y)$.
The mean curvature at the boundary $H_{p}$, where $p\in\partial M$
is the trace of the second fundamental form.

If we consider Fermi coordinates in a neighborhood of $p$, and we
note by the matrix $(h_{ij})_{i,j=1,\dots,n-1}$ the second fundamental
form, we have the well known formulas
\begin{eqnarray}
g^{ij}(y) & = & \delta_{ij}+2h_{ij}(0)y_{n}+O(|y|^{2})\text{ for }i,j=1,\dots n-1\label{eq:g1}\\
g^{in}(y) & = & \delta_{in}\label{eq:g2}\\
\sqrt{g}(y) & = & 1-(n-1)H(0)y_{n}+O(|y|^{2})\label{eq:g3}
\end{eqnarray}
where $(y_{1},\dots,y_{n})$ are the Fermi coordinates and, by definition
of $h_{ij}$, 
\begin{equation}
H=\frac{1}{n-1}\sum_{i}^{n-1}h_{ii}.\label{eq:H}
\end{equation}
Also, by Escobar \cite[eq. (3.2)]{Es92}, we have that 
\begin{equation}
\left.\frac{\partial^{2}}{\partial y_{n}\partial y_{i}}\sqrt{g}(y)\right|_{y=0}=-(n-1)\frac{\partial H}{\partial y_{i}}(0)\text{ for }i=1,\dots,n-1\label{eq:G}
\end{equation}


It is well known that, in $\mathbb{R}^{n}$, there is a unique positive
radially symmetric function $V(y)\in H^{1}(\mathbb{R}^{n})$ satisfying
\begin{equation}
-\Delta V+V=V^{p-1}\text{ on }\mathbb{R}^{n}.\label{eq:rn}
\end{equation}
 Moreover, the function $V$ exponentially decays at infinity as well
as its derivative, that is, for some $c>0$ 
\begin{eqnarray*}
\lim_{|y|\rightarrow\infty}V(|y|)|y|^{\frac{n-1}{2}}e^{|y|}=c &  & \lim_{|y|\rightarrow\infty}V'(|y|)|y|^{\frac{n-1}{2}}e^{|y|}=-c.
\end{eqnarray*}
We can define on the half space $\mathbb{R}_{+}^{n}=\left\{ (y_{1,}\dots,y_{n})\in\mathbb{R}^{n}\ ,\ y_{n}\ge0\right\} $
the function
\[
U(y)=\left.V\right|_{y_{n}\ge0}.
\]
 The function $U$ satisfies the following Neumann problem in $\mathbb{R}_{+}^{n}$
\begin{equation}
\left\{ \begin{array}{cc}
-\Delta U+U=U^{p-1} & \text{in }\mathbb{R}_{+}^{n}\\
{\displaystyle \frac{\partial U}{\partial y_{n}}=0} & \text{on }\left\{ y_{n}=0\right\} .
\end{array}\right.\label{eq:P-Rn}
\end{equation}
We set $U_{\varepsilon}(y)=U\left(\frac{y}{\varepsilon}\right)$.
\begin{lem}
The space solution of the linearized problem
\begin{equation}
\left\{ \begin{array}{cc}
-\Delta\varphi+\varphi=(p-1)U^{p-2}\varphi & \text{ in }\mathbb{R}_{+}^{n}\\
{\displaystyle \frac{\partial\varphi}{\partial y_{n}}=0} & \text{on }\left\{ y_{n}=0\right\} .
\end{array}\right.\label{eq:linear}
\end{equation}
 is generated by the linear combination of 
\[
\varphi^{i}=\frac{\partial U}{\partial y_{i}}(y)\text{ for }i=1,\dots,n-1.
\]
\end{lem}
\begin{proof}
It is trivial that every linear combination of $\varphi^{i}$ is a
solution of (\ref{eq:linear}). We notice that ${\displaystyle \frac{\partial U}{\partial y_{n}}}$
is not a solution of (\ref{eq:linear}) because the derivative on
$\{y_{n}=0\}$ is not zero. 

For the converse, suppose $\bar{\varphi}(y)$ be a solution of (\ref{eq:linear}).
Then, by even reflection around $y_n$, we can construct a solution $\tilde{\varphi}$ of
\begin{equation}
-\Delta\tilde\varphi+\tilde\varphi=(p-1)U^{p-2}\tilde\varphi \text{ in }\mathbb{R}^{n}
\label{eq:linRN}
\end{equation}
 with ${\displaystyle \frac{\partial\tilde{\varphi}}{\partial y_{n}}=0}$
on $y_{n}=0$. But all solution of (\ref{eq:linRN}) with zero derivative
on $y_{n}=0$ are linear combination of ${\displaystyle {\displaystyle \frac{\partial V}{\partial y_{j}}}}$
with $j=1,\cdots,n-1$. 
\end{proof}

We endow $H^{1}(M)$ with the scalar product ${\displaystyle \left\langle u,v\right\rangle _{\varepsilon}:=\frac{1}{\varepsilon^{n}}\int_{M}\varepsilon^{2}g(\nabla u,\nabla v)+uvd\mu_{g}}$
and the norm $\|u\|_{\varepsilon}=\left\langle u,u\right\rangle _{\varepsilon}^{1/2}$.
We call $H_{\varepsilon}$ the space $H^{1}$ equipped with the
norm $\|\cdot\|_{\varepsilon}$. We also define $L_{\varepsilon}^{p}$
as the space $L^{p}(M)$ endowed with the norm ${\displaystyle |u|_{\varepsilon,p}=\left(\frac{1}{\varepsilon^{n}}\int_{M}u^{p}d\mu_{g}\right)^{1/p}}$. 

For any $p\in[2,2^{*})$ if $n\ge3$ or for all $p\ge2$ if $n=2$,
the embedding $i_{\varepsilon}:H_{\varepsilon}\hookrightarrow L_{\varepsilon,p}$
is a compact, continuous map, and it holds $|u|_{\varepsilon,p}\le c\|u\|_{\varepsilon}$
for some constant $c$ not depending on $\varepsilon$. We define
the adjoint operator $i_{\varepsilon}^{*}:L_{\varepsilon,p'}:\hookrightarrow H_{\varepsilon}$
as 
\[
u=i_{\varepsilon}^{*}(v)\ \Leftrightarrow\ \left\langle u,\varphi\right\rangle _{\varepsilon}=\frac{1}{\varepsilon^{n}}\int_{M}v\varphi d\mu_{g},
\]
 so we can rewrite problem (\ref{eq:P}) in an equivalent formulation
\[
u=i_{\varepsilon}^{*}\left(\left(u^{+}\right)^{p-1}\right).
\]

\begin{rem}
\label{rem:ieps}We have that $\|i_{\varepsilon}^{*}(v)\|_{\varepsilon}\le c|v|_{p',\varepsilon}$
\end{rem}
From now on we set, for sake of simplicity
\[
f(u)=(u^{+})^{p-1}\text{ and }f'(u)=(p-1)(u^{+})^{p-2}
\]
We want to split the space $H_{\varepsilon}$ in a finite dimensional
space generated by the solution of (\ref{eq:linear}) and its orthogonal
complement. Fixed $\xi\in\partial M$ and $R>0$, we consider on the
manifold the functions 
\begin{equation}
Z_{\varepsilon,\xi}^{i}=\left\{ \begin{array}{ccc}
\varphi_{\varepsilon}^{i}\left(\left(\psi_{\xi}^{\partial}\right)^{-1}(x)\right)\chi_{R}\left(\left(\psi_{\xi}^{\partial}\right)^{-1}(x)\right) &  & x\in I_{\xi}(R):=I_{\xi}(R,R);\\
0 &  & \text{elsewhere}.
\end{array}\right.\label{eq:Zi}
\end{equation}
where ${\displaystyle \varphi_{\varepsilon}^{i}(y)=\varphi^{i}\left(\frac{y}{\varepsilon}\right)}$
and $\chi_{R}:B^{n-1}(0,R)\times[0,R)\rightarrow\mathbb{R}^{+}$ is
a smooth cut off function such that $\chi_{R}\equiv1$ on $B^{n-1}(0,R/2)\times[0,R/2)$
and $|\nabla\chi|\le2$.

In the following, for sake of simplicity, we denote
\begin{equation}
D^{+}(R)=B^{n-1}(0,R)\times[0,R)\label{eq:D+}
\end{equation}

Let 
\[
K_{\varepsilon,\xi}:=\mbox{Span}\left\{ Z_{\varepsilon,\xi}^{1},\cdots,Z_{\varepsilon,\xi}^{n-1}\right\} .
\]
 We can split $H_{\varepsilon}$ in the sum of the $\left(n-1\right)$-dimensional
space and its orthogonal complement with respect of $\left\langle \cdot,\cdot\right\rangle _{\varepsilon}$,
i.e.
\[
K_{\varepsilon,\xi}^{\bot}:=\left\{ u\in H_{\varepsilon}\ ,\ \left\langle u,Z_{\varepsilon,\xi}^{i}\right\rangle _{\varepsilon}=0.\right\} .
\]
We solve problem (\ref{eq:P}) by a Lyapunov Schmidt reduction: defined
\[
W_{\varepsilon,\xi}(x)=\left\{ \begin{array}{ccc}
U_{\varepsilon}\left(\left(\psi_{\xi}^{\partial}\right)^{-1}(x)\right)\chi_{R}\left(\left(\psi_{\xi}^{\partial}\right)^{-1}(x)\right) &  & x\in I_{\xi}(R):=I_{\xi}(R,R);\\
0 &  & \text{elsewhere}.
\end{array}\right.
\]
we look for a function of the form $W_{\varepsilon,\xi}+\phi$ with
$\phi\in K_{\varepsilon,\xi}^{\bot}$ such that 
\begin{eqnarray}
\Pi_{\varepsilon,\xi}^{\bot}\left\{ W_{\varepsilon,\xi}+\phi-i_{\varepsilon}^{*}\left[f\left(W_{\varepsilon,\xi}+\phi\right)\right]\right\}  & = & 0\label{eq:red1}\\
\Pi_{\varepsilon,\xi}\left\{ W_{\varepsilon,\xi}+\phi-i_{\varepsilon}^{*}\left[f\left(W_{\varepsilon,\xi}+\phi\right)\right]\right\}  & = & 0\label{eq:red2}
\end{eqnarray}
where $\Pi_{\varepsilon,\xi}:H_{\varepsilon}\rightarrow K_{\varepsilon,\xi}$
and $\Pi_{\varepsilon,\xi}^{\bot}:H_{\varepsilon}\rightarrow K_{\varepsilon,\xi}^{\bot}$
are, respectively, the projection on $K_{\varepsilon,\xi}$ and $K_{\varepsilon,\xi}^{\bot}$.
We see that $W_{\varepsilon,\xi}+\phi$ is a solution of (\ref{eq:P})
if and only if $W_{\varepsilon,\xi}+\phi$ solves (\ref{eq:red1}-\ref{eq:red2}).

Hereafter we collect a series of results which will be useful in the paper.

\begin{defn}
\label{def:Estorto}Given $\xi_{0}\in\partial M$, using the normal
coordinates on the sub manifold $\partial M$, we define
\[
\mathcal{E}(y,x)=\left(\exp_{\xi(y)}^{\partial}\right)^{-1}(x)=\left(\exp_{\exp_{\xi_{0}}^{\partial}y}^{\partial}\right)^{-1}(\exp_{\xi_{0}}^{\partial}\bar{\eta})=\tilde{\mathcal{E}}(y,\bar{\eta})
\]
where $x,\xi(y)\in\partial M$, $y,\bar{\eta}\in B(0,R)\subset\mathbb{R}^{n-1}$
and $\xi(y)=\exp_{\xi_{0}}^{\partial}y$, $x=\exp_{\xi_{0}}^{\partial}\bar{\eta}$.
$ $Using Fermi coordinates around $\xi_{0}$ in a similar way we
define
\[
\mathcal{H}(y,x)=\left(\psi_{\xi(y)}^{\partial}\right)^{-1}(x)=\left(\psi_{\exp_{\xi_{0}}^{\partial}y}^{\partial}\right)^{-1}\left(\psi_{\xi_{0}}^{\partial}(\bar{\eta},\eta_{n})\right)=\tilde{\mathcal{H}}(y,\bar{\eta},\eta_{n})=(\tilde{\mathcal{E}}(y,\bar{\eta}),\eta_{n})
\]
where $x\in M$, $\eta=(\bar{\eta},\eta_{n})$, with $\bar{\eta}\in B(0,R)\subset\mathbb{R}^{n-1}$
and $0\le\eta_{n}<R$, $\xi(y)=\exp_{\xi_{0}}^{\partial}y\in\partial M$
and $x=\psi_{\xi_{0}}^{\partial}(\eta)$. \end{defn}

\begin{lem}
\label{lem:derWeps}Set $x=\psi_{\xi_{0}}^{\partial}(\varepsilon z)$
where $z=(\bar{z},z_{n})$ and $ $$\xi(y)=\exp_{\xi_{0}}^{\partial}(y)$,
for $j=1,\dots,n-1$ we have 
\[
\left.\frac{\partial}{\partial y_{j}}W_{\varepsilon,\xi(y)}(x)\right|_{y=0}=\sum_{k=1}^{n-1}\left[\frac{1}{\varepsilon}\chi_{R}(\varepsilon z)\frac{\partial}{\partial z_{k}}U(z)+U(z)\frac{\partial}{\partial z_{k}}\chi_{R}(\varepsilon z)\right]\left.\frac{\partial}{\partial y_{j}}\tilde{\mathcal{E}_{k}}(y,\varepsilon\bar{z})\right|_{y=0}.
\]
\end{lem}
We need some preliminaries in order to prove of Lemma  \ref{lem:derWeps}.

\begin{lem}
\label{lem:Estorto}It holds
\begin{align*}
{\mathcal{E}}(0,\bar{\eta})= & \bar{\eta}\text{ for }\bar{\eta}\in\mathbb{R}^{n-1}\\
\frac{\partial\tilde{\mathcal{E}}_{k}}{\partial\eta_{j}}(0,\bar\eta)= & \delta_{jk}\text{ for }y\in\mathbb{R}^{n-1},\ j,k=1,\dots,n-1\\
\frac{\partial\tilde{\mathcal{E}}_{k}}{\partial y_{j}}(0,0)= & -\delta_{jk}\text{ for }j,k=1,\dots,n-1\\
\frac{\partial^{2}\tilde{\mathcal{E}}_{k}}{\partial y_{j}\partial\eta_{h}}(0,0)= & 0\text{ for }j,h,k=1,\dots,n-1
\end{align*}
\end{lem}
\begin{proof}
We recall that $\tilde{\mathcal{E}}(y,\bar{\eta})=\left(\exp_{\xi(y)}^{\partial}\right)^{-1}(\exp_{\xi_{0}}^{\partial}\bar{\eta})$,
so the first claim is obvious. Let us introduce, for $y,\bar{\eta}\in B(0,R)\subset\mathbb{R}^{n-1}$
\begin{align*}
F(y,\bar{\eta}) & =\left(\exp_{\xi_{0}}^{\partial}\right)^{-1}\left(\exp_{\xi(y)}^{\partial}(\bar{\eta})\right)\\
\Gamma(y,\bar{\eta}) & =\left(y,F(y,\bar{\eta})\right).
\end{align*}
We notice that $\Gamma^{-1}(y,\beta),=(y,\tilde{\mathcal{E}}(y,\beta))$.
We can easily compute the derivative of $\Gamma$. Given  $\hat{y},\hat{\eta}\in \mathbb{R}^{n-1}$ we have
\[
\Gamma'(\hat{y},\hat{\eta})[{y},{\beta}]=\left(\begin{array}{cc}
\text{Id}_{\mathbb{R}^{n-1}} & 0\\
F_{y}'(\hat{y},\hat{\eta}) & F_{\eta}'(\hat{y},\hat{\eta})
\end{array}\right)\left(\begin{array}{c}
{y}\\
{\beta}
\end{array}\right),
\]
thus
\[
\left(\Gamma^{-1}\right)^{\prime}(\hat{y},\hat{\eta})[{y}, {\beta}]=\left(\begin{array}{cc}
\text{Id}_{\mathbb{R}^{n-1}} & 0\\
-\left(F_{\eta}'(\hat{y},\hat{\eta})\right)^{-1}F_{y}'(\hat{y},\hat{\eta}) & \left(F_{\eta}'(\hat{y},\hat{\eta})\right)^{-1}
\end{array}\right)\left(\begin{array}{c}
 {y}\\
 {\beta}
\end{array}\right)
\]
Here $y,\beta\in \mathbb{R}^{n-1}$. Now, by direct computation we have that
\[
F_{\eta}'(0,\hat{\eta})=\text{Id}_{\mathbb{R}^{n-1}}\text{ and }F_{y}'(\hat{y},0)=\text{Id}_{\mathbb{R}^{n-1}},
\]
so $\frac{\partial\tilde{\mathcal{E}}_{k}}{\partial\eta_{j}}(0,\hat\eta)=
\left(\left(F_{\eta}'(0,\hat{\eta})\right)^{-1}\right)_{jk}=\delta_{jk}$
and $\frac{\partial\tilde{\mathcal{E}}_{k}}{\partial y_{j}}(0,0)=
\left(-\left(F_{\eta}'(0,0)\right)^{-1}F_{y}'(0,0)\right)_{jk}=-\delta_{jk}$.
For the last claim we refer to \cite[Lemma 6.4]{MP09}\end{proof}
\begin{lem}
\label{lem:Estorto2}We have that 
\begin{align*}
\tilde{\mathcal{H}}(0,\bar{\eta},\eta_{n})= 
& (\bar{\eta},\eta_{n})\text{ for }\bar{\eta}\in\mathbb{R}^{n-1},\eta_{n}\in\mathbb{R}_{+}\\
\frac{\partial\tilde{\mathcal{H}}_{k}}{\partial y_{j}}(0,0,\eta_{n})= 
& -\delta_{jk}\text{ for }j,k=1,\dots,n-1,\eta_{n}\in\mathbb{R}_{+}\\
\frac{\partial\tilde{\mathcal{H}}_{n}}{\partial y_{j}}(y,\bar{\eta},\eta_{n})= 
& 0\text{ for }j=1,\dots,n-1,y,\bar{\eta}\in\mathbb{R}^{n-1},\eta_{n}\in\mathbb{R}_{+}\\
\frac{\partial\tilde{\mathcal{H}}_{k}}{\partial\eta_{n}}(y,\bar{\eta},\eta_{n})= 
& 0\text{ for }j,k=1,\dots,n-1,\bar{\eta}\in\mathbb{R}^{n-1},\eta_{n}\in\mathbb{R}_{+}\\
\frac{\partial^{2}\tilde{\mathcal{H}}_{k}}{\partial\eta_{n}\partial y_{j}}(y,\bar{\eta},\eta_{n})= 
& 0\text{ for }j,k=1,\dots,n-1,\bar{\eta}\in\mathbb{R}^{n-1},\eta_{n}\in\mathbb{R}_{+}
\end{align*}
\end{lem}
\begin{proof}
The first three claim follows immediately by Definition \ref{def:Estorto}
and Lemma \ref{lem:Estorto}. For the last two claims, observe that
$\tilde{\mathcal{H}}_{k}(y,\bar{\eta},\eta_{n})=\tilde{\mathcal{E}}_{k}(y,\bar{\eta})$
which does not depends on $\eta_{n}$ as well as its derivatives.
\end{proof}

We now prove the claimed result.
\begin{proof}[Proof of Lemma \ref{lem:derWeps}.]
By definition \ref{def:Estorto}, set $x=\psi_{\xi_{0}}^{\partial}(\eta)=\psi_{\xi_{0}}^{\partial}(\bar{\eta},\eta_{n})$ with 
$\eta=(\bar\eta,\eta_n)\in \mathbb{R}^n$,
and $\xi(y)=\exp_{\xi_{0}}^{\partial}(y)$ where $ y\in \mathbb{R}^{n-1}$, we have 
that 
$$W_{\varepsilon,\xi(y)}(x)=U\left(\frac{\tilde{\mathcal{H}}(y,\eta)}{\varepsilon}\right)\chi_{R}(\tilde{\mathcal{H}}(y,\eta)).$$
Fixed $j$, by Lemma \ref{lem:Estorto2}, 
\begin{align*}
\left.\frac{\partial}{\partial y_{j}}W_{\varepsilon,\xi(y)}(x)\right|_{y=0} &
 =\sum_{k=1}^{n}
 \left.
 \frac{\partial}{\partial v_{k}}\left[\chi_{R}(\tilde{\mathcal{H}}(y,\eta))
 U_{\varepsilon}(\tilde{\mathcal{H}}(y,\eta))\right]
 \right|_{\tilde{\mathcal{H}}(0,\eta)}
 \left.\frac{\partial}{\partial y_{i}}\tilde{\mathcal{H}_{k}}(y,\eta)\right|_{y=0}\\
 & =\sum_{k=1}^{n-1}\frac{\partial}{\partial\eta_{k}}\left[\chi_{R}(\eta)U_{\varepsilon}(\eta)\right]\left.\frac{\partial}{\partial y_{j}}\tilde{\mathcal{E}_{k}}(y,\bar{\eta})\right|_{y=0}\\
 &
  =\sum_{k=1}^{n-1}\frac{\partial}{\partial z_{k}}\left[\chi_{R}(\varepsilon z)U(z)\right]\left.\frac{\partial}{\partial y_{j}}
  \tilde{\mathcal{E}_{k}}(y,\varepsilon\bar{z})\right|_{y=0}
\end{align*}
Because $\tilde{\mathcal{H}_{k}}(y,\bar\eta,\eta_n)= {\mathcal{E}_{k}}(y,\bar{\eta})$ for $k=1,\dots, n-1$. Using the
 change of variables $\eta=\varepsilon z=(\varepsilon \bar z,\varepsilon z_n)$, we get the claim.
\end{proof}

\section{\label{sec:Reduction}Reduction to finite dimensional space}

In this section we find a solution for equation (\ref{eq:red1}).
In particular, we prove that for all $\varepsilon>0$ and for all
$\xi\in\partial M$ there exists $\phi_{\varepsilon,\xi}\in K_{\varepsilon,\xi}^{\bot}$
solving (\ref{eq:red1}). Here and in the hereafter, all the proof
are similar to \cite{MP09}. So, for the sake of simplicity, we will
underline the parts where differences appear, and sketch the remains
of the proofs (we will provide precise references for each proof).

We introduce the linear operator $L_{\varepsilon,\xi}:K_{\varepsilon,\xi}^{\bot}\rightarrow K_{\varepsilon,\xi}^{\bot}$
\[
L_{\varepsilon,\xi}(\phi):=\Pi_{\varepsilon,\xi}^{\bot}\left\{ \phi-i_{\varepsilon}^{*}\left[f'(W_{\varepsilon,\xi})\phi\right]\right\} 
\]
thus we can rewrite equation (\ref{eq:red1}) as 
\[
L_{\varepsilon,\xi}(\phi)=N_{\varepsilon,\xi}(\phi)+R_{\varepsilon,\xi}
\]
 where $N_{\varepsilon,\xi}(\phi)$ is the nonlinear term 
\[
N_{\varepsilon,\xi}:=\Pi_{\varepsilon,\xi}^{\bot}\left\{ i_{\varepsilon}^{*}\left[f(W_{\varepsilon,\xi}+\phi)-f(W_{\varepsilon,\xi})-f'(W_{\varepsilon,\xi})\phi\right]\right\} 
\]
and $R_{\varepsilon,\xi}$ is a remainder term
\[
R_{\varepsilon,\xi}:=\Pi_{\varepsilon,\xi}^{\bot}\left\{ i_{\varepsilon}^{*}\left[f(W_{\varepsilon,\xi})\right]-W_{\varepsilon,\xi}\right\} .
\]
The first step is to prove that the linear term is invertible.
\begin{lem}
\label{lem:Linv}There exist $\varepsilon_{0}$ and $c>0$ such that,
for any $\xi\in\partial M$ and $\varepsilon\in(0,\varepsilon_{0})$
\[
\|L_{\varepsilon,\xi}\|_{\varepsilon}\geq c\|\phi\|_{\varepsilon}\text{ for any }\phi\in K_{\varepsilon,\xi}^{\bot}.
\]
\end{lem}

The proof of this Lemma is postponed to the Appendix. We estimate now the remainder term $R_{\varepsilon,\xi}$.
\begin{lem}
\label{lem:Reps}There exists $\varepsilon_{0}>0$ and $c>0$ such
that for any $\xi\in\partial M$ and for all $\varepsilon\in(0,\varepsilon_{0})$
it holds 
\[
\|R_{\varepsilon,\xi}\|_{\varepsilon}\le c\varepsilon^{1+\frac{n}{p'}}.
\]
\end{lem}
\begin{proof}
We proceed as in \cite[Lemma 3.3]{MP09}. We define on $M$ the function
$V_{\varepsilon,\xi}$ such that $W_{\varepsilon,\xi}=i_{\varepsilon}^{*}(V_{\varepsilon,\xi})$,
thus $-\varepsilon^{2}\Delta_{g}W_{\varepsilon,\xi}+W_{\varepsilon,\xi}=V_{\varepsilon,\xi}$. 

It is well known%
\footnote{\cite[page134]{Mor}%
}, by definition of Laplace-Beltrami operator, that in a local chart
it holds 
\[
-\Delta v=-\Delta_{g}v+(g_{\xi}^{ij}-\delta_{ij})\frac{\partial^{2}}{\partial x_{i}\partial x_{j}}v-g_{\xi}^{ij}\Gamma_{ij}^{k}\frac{\partial}{\partial x_{k}}v
\]
where $\Delta$ is the euclidean Laplace operator. Thus, defined 
\[
\tilde{V}_{\varepsilon,\xi}(y)=V_{\varepsilon,\xi}\left(\psi_{\xi}^{\partial}(y)\right),\ y\in D^{+}(R)
\]
we have 
\begin{eqnarray}
\tilde{V}_{\varepsilon,\xi}(y) & = & -\varepsilon^{2}\Delta_{g}(U_{\varepsilon}\chi_{R})+U_{\varepsilon}\chi_{R}=\nonumber \\
 & = & U_{\varepsilon}^{p-1}\chi_{R}-\varepsilon^{2}U_{\varepsilon}\Delta\chi_{R}-2\varepsilon^{2}\nabla U_{\varepsilon}\nabla\chi_{R}\label{eq:Veps}\\
 &  & -\varepsilon^{2}(g_{\xi}^{ij}-\delta_{ij})\frac{\partial^{2}}{\partial y_{i}\partial y_{j}}(U_{\varepsilon}\chi_{R})+\varepsilon^{2}g_{\xi}^{ij}\Gamma_{ij}^{k}\frac{\partial}{\partial y_{k}}(U_{\varepsilon}\chi_{R})\nonumber 
\end{eqnarray}
Also, we remind that, by Remark \ref{rem:ieps} and by definition
of $R_{\varepsilon,\xi}$, it holds
\[
\|R_{\varepsilon,\xi}\|_{\varepsilon}\le\|i_{\varepsilon}^{*}f(W_{\varepsilon,\xi})-W_{\varepsilon,\xi}\|_{\varepsilon}\le c\left|W_{\varepsilon,\xi}^{p-1}-V_{\varepsilon,\xi}\right|_{p',\varepsilon}.
\]
Finally, by definition of $W_{\varepsilon,\xi}$ and by (\ref{eq:Veps})
we get
\begin{align*}
\left|W_{\varepsilon,\xi}^{p-1}-V_{\varepsilon,\xi}\right|_{p',\varepsilon}^{p'}= & \int_{D^{+}(R)}\left|U_{\varepsilon}^{p-1}(y)\chi_{R}^{p-1}(y)-\tilde{V}_{\varepsilon,\xi}(y)\right|^{p'}|g_{\xi}(y)|^{1/2}dy\\
\le & c\int_{D^{+}(R)}\left|U_{\varepsilon}^{p-1}(y)\left(\chi_{R}^{p-1}(y)-\chi_{R}(y)\right)\right|^{p'}dy\\
 & +c\varepsilon^{2p'}\int_{D^{+}(R)}U_{\varepsilon}^{p'}\left|\Delta\chi_{R}\right|^{p'}dy+c\varepsilon^{2p'}\int_{D^{+}(R)}\left|\nabla U_{\varepsilon}\cdot\nabla\chi_{R}\right|^{p'}dy\\
 & +c\varepsilon^{2p'}\int_{D^{+}(R)}\left|g_{\xi}^{ij}(y)-\delta_{ij}\right|^{p'}\left|\frac{\partial^{2}}{\partial y_{i}\partial y_{j}}(U_{\varepsilon}\chi_{R})(y)\right|^{p'}dy\\
 & +c\varepsilon^{2p'}\int_{D^{+}(R)}\left|g_{\xi}^{ij}(y)\Gamma_{ij}^{k}(y)\right|^{p'}\left|\frac{\partial}{\partial y_{k}}(U_{\varepsilon}\chi_{R})(y)\right|^{p'}dy
\end{align*}
By exponential decay and by definition of $\chi_{r}$, using (\ref{eq:g1})
and (\ref{eq:g2}) we have 
\begin{multline*}
\varepsilon^{2p'}\int_{D^{+}(R)}\left|g_{\xi}^{ij}(y)-\delta_{ij}\right|^{p'}\left|\frac{\partial^{2}}{\partial y_{i}\partial y_{j}}(U_{\varepsilon}\chi_{R})(y)\right|^{p'}dy\\
=\varepsilon^{2p'}\int_{D^{+}(R)}\left|g_{\xi}^{ij}(y)-\delta_{ij}\right|^{p'}\left|\frac{\partial^{2}}{\partial y_{i}\partial y_{j}}U_{\varepsilon}(y)\right|^{p'}dy+O(\varepsilon^{n+p'})\\
\le\varepsilon^{n}\int_{\mathbb{R}_{+}^{n}}\left|g_{\xi}^{ij}(\varepsilon z)-\delta_{ij}\right|^{p'}\left|\frac{\partial^{2}}{\partial z_{i}\partial z_{j}}U(z)\right|^{p'}dz+O(\varepsilon^{n+p'})\\
\le\varepsilon^{n+p'}\int_{\mathbb{R}_{+}^{n}}\left|\frac{\partial^{2}}{\partial z_{i}\partial z_{j}}U(z)\right|^{p'}dz+O(\varepsilon^{n+p'})=O(\varepsilon^{n+p'}).
\end{multline*}
The other terms can be estimate in a similar way.
\end{proof}
By fixed point theorem and by implicit function theorem we can solve
equation (\ref{eq:red1}). 
\begin{prop}
\label{prop:phieps}There exists $\varepsilon_{0}>0$ and $c>0$ such
that for any $\xi\in\partial M$ and for all $\varepsilon\in(0,\varepsilon_{0})$
there exists a unique $\phi_{\varepsilon,\xi}=\phi(\varepsilon,\xi)\in K_{\varepsilon,\xi}^{\bot}$
which solves (\ref{eq:red1}). Moreover
\[
\|\phi_{\varepsilon,\xi}\|_{\varepsilon}<c\varepsilon^{1+\frac{n}{p'}}.
\]
Finally, $\xi\mapsto\phi_{\varepsilon,\xi}$ is a $C^{1}$ map.\end{prop}
\begin{proof}
The proof is similar to Proposition 3.5 of \cite{MP09}, which we
refer to for all details. We want to solve (\ref{eq:red1}) by a fixed
point argument. We define the operator
\begin{eqnarray*}
T_{\varepsilon,\xi} & : & K_{\varepsilon,\xi}^{\bot}\rightarrow K_{\varepsilon,\xi}^{\bot}\\
T_{\varepsilon,\xi}(\phi) & = & L_{\varepsilon,\xi}^{-1}\left(N_{\varepsilon,\xi}(\phi)+R_{\varepsilon,\xi}\right)
\end{eqnarray*}
By Lemma \ref{lem:Linv}, $T_{\varepsilon,\xi}$ is well defined and
it holds
\begin{eqnarray*}
\|T_{\varepsilon,\xi}(\phi)\|_{\varepsilon} & \le & c\left(\|N_{\varepsilon,\xi}(\phi)\|_{\varepsilon}+\|R_{\varepsilon,\xi}\|_{\varepsilon}\right)\\
\|T_{\varepsilon,\xi}(\phi_{1})-T_{\varepsilon,\xi}(\phi_{2})\|_{\varepsilon} & \le & c\left(\|N_{\varepsilon,\xi}(\phi_{1})-N_{\varepsilon,\xi}(\phi_{2})\|_{\varepsilon}\right)
\end{eqnarray*}
for some suitable constant $c>0$. By the mean value theorem (and
by the properties of $i^{*}$) we get
\[
\|N_{\varepsilon,\xi}(\phi_{1})-N_{\varepsilon,\xi}(\phi_{2})\|_{\varepsilon}\le c\left|f'(W_{\varepsilon,\xi}+\phi_{2}+t(\phi_{1}-\phi_{2}))-f'(W_{\varepsilon,\xi})\right|_{\frac{p}{p-2},\varepsilon}\|\phi_{1}-\phi_{2}\|_{\varepsilon}.
\]
By \cite[Remark 3.4]{MP09}, we have that $\left|f'(W_{\varepsilon,\xi}+\phi_{2}+t(\phi_{1}-\phi_{2}))-f'(W_{\varepsilon,\xi})\right|_{\frac{p}{p-2},\varepsilon}<<1$
provided $\|\phi_{1}\|_{\varepsilon}$ and $\|\phi_{2}\|_{\varepsilon}$
small enough. Thus there exists $0<C<1$ such that $\|T_{\varepsilon,\xi}(\phi_{1})-T_{\varepsilon,\xi}(\phi_{2})\|_{\varepsilon}\le C\|\phi_{1}-\phi_{2}\|_{\varepsilon}$.
Also, with the same estimates we get 
\[
\|N_{\varepsilon,\xi}(\phi)\|_{\varepsilon}\le c\left(\|\phi\|_{\varepsilon}^{2}+\|\phi\|_{\varepsilon}^{p-1}\right).
\]
This, combined with Lemma \ref{lem:Reps} gives us 
\[
\|T_{\varepsilon,\xi}(\phi)\|_{\varepsilon}\le c\left(\|N_{\varepsilon,\xi}(\phi)\|_{\varepsilon}+\|R_{\varepsilon,\xi}\|_{\varepsilon}\right)\le c\left(\|\phi\|_{\varepsilon}^{2}+\|\phi\|_{\varepsilon}^{p-1}+\varepsilon^{1+\frac{n}{p'}}\right).
\]
 So, there exists $c>0$ such that $T_{\varepsilon,\xi}$ maps a ball
of center $0$ and radius $c\varepsilon^{1+\frac{n}{p'}}$ in $K_{\varepsilon,\xi}^{\bot}$
into itself and it is a contraction. So there exists a fixed point
$\phi_{\varepsilon,\xi}$ with norm $\|\phi_{\varepsilon,\xi}\|_{\varepsilon}\le\varepsilon^{1+\frac{n}{p'}}$. 

The regularity of $ $$\phi_{\varepsilon,\xi}$ with respect to $\xi$
is proved via implicit function theorem. Let us define the functional
\begin{eqnarray*}
G & : & \partial M\times H_{\varepsilon}\rightarrow\mathbb{R}\\
G(\xi,u) & := & \Pi_{\varepsilon,\xi}^{\bot}\left\{ W_{\varepsilon,\xi}+\Pi_{\varepsilon,\xi}^{\bot}u+i_{\varepsilon}^{*}\left[f\left(W_{\varepsilon,\xi}+\Pi_{\varepsilon,\xi}^{\bot}u\right)\right]\right\} +\Pi_{\varepsilon,\xi}u.
\end{eqnarray*}
We have that $G(\xi,\phi_{\varepsilon,\xi})=0$ and that the operator
${\displaystyle \frac{\partial}{\partial u}G(\xi,\phi_{\varepsilon,\xi}):H_{\varepsilon}\rightarrow H_{\varepsilon}}$
is invertible. This concludes the proof.
\end{proof}

\section{Sketch of the proof of Theorem \ref{thm:main}}\label{proof}

In section \ref{sec:Reduction}, Proposition \ref{prop:phieps} we
found a function $\phi_{\varepsilon,\xi}$ solving (\ref{eq:red1}).
In order to solve (\ref{eq:red2}) we define the functional $J_{\varepsilon}:H^{1}(M)\rightarrow\mathbb{R}$
\[
J_{\varepsilon}(u)=\frac{1}{\varepsilon^{n}}\int_{M}\frac{1}{2}\varepsilon^{2}|\nabla u|_{g}^{2}+\frac{1}{2}u^{2}-\frac{1}{p}(u^{+})^{p}d\mu_{g}.
\]
In which follows we will often use the notation $F(u)=\frac{1}{p}(u^{+})^{p}$. 

By $J_{\varepsilon}$ we define the reduced functional $\tilde{J}_{\varepsilon}$
on $\partial M$ as 
\[
\tilde{J}_{\varepsilon}(\xi)=J_{\varepsilon}(W_{\varepsilon,\xi}+\phi_{\varepsilon,\xi})
\]
where $\phi_{\varepsilon,\xi}$ is uniquely determined by Proposition
\ref{prop:phieps}.
\begin{rem}
Our goal is to find critical points for $\tilde{J}_{\varepsilon}$,
since any critical point $\xi$ for $\tilde{J}_{\varepsilon}$ corresponds
to a function $\phi_{\varepsilon,\xi}+W_{\varepsilon,\xi}$ which
solves equation (\ref{eq:red2}).
\end{rem}
At this point we give the expansion for the functional $\tilde{J}_{\varepsilon}$
with respect to $\varepsilon$. By Lemma \ref{lem:espJ1} and Lemma
\ref{lem:espJ2} it holds 
\begin{equation}
\tilde{J}_{\varepsilon}(\xi)=C-\varepsilon H(\xi)+o(\varepsilon)\label{eq:espansioneJ}
\end{equation}
$C^{1}$ uniformly with respect to $\xi\in\partial M$ as $\varepsilon$
goes to zero. Here $H(\xi)$ is the mean curvature of the boundary
$\partial M$ at $\xi$. If $\xi_{0}$ is a $C^{1}$-stable critical
point for $H$, in light of (\ref{eq:espansioneJ}) and by definition
of $C^{1}$-stability, we have that, for $\varepsilon$ small enough
there exists $\xi_{\varepsilon}$ close to $\xi_{0}$ critical point
for $\tilde{J}_{\varepsilon}$, and we can prove Theorem \ref{thm:main}.

\section{Asymptotic expansion of the reduced functional}\label{expansion}

In this  we study the asymptotic expansion of $\tilde{J}_{\varepsilon}(\xi)$
with respect to $\varepsilon$. 
\begin{lem}
\label{lem:espJ1}It holds 
\begin{equation}
\tilde{J}_{\varepsilon}(\xi)=J_{\varepsilon}(W_{\varepsilon,\xi}+\phi_{\varepsilon,\xi})=J_{\varepsilon}(W_{\varepsilon,\xi})+o(\varepsilon)\label{eq:asexp1}
\end{equation}
uniformly with respect to $\xi\in\partial M$ as $\varepsilon$ goes
to zero.

Moreover, setting $\xi(y)=\exp_{\xi}^{\partial}(y)$, $y\in B^{n-1}(0,r)$
it holds
\begin{eqnarray}
\left(\frac{\partial}{\partial y_{h}}\tilde{J}_{\varepsilon}(\xi(y))\right)_{|_{y=0}} & = & \left(\frac{\partial}{\partial y_{h}}J_{\varepsilon}(W_{\varepsilon,\xi(y)}+\phi_{\varepsilon,\xi(y)})\right)_{|_{y=0}}=\nonumber \\
 & = & \left(\frac{\partial}{\partial y_{h}}J_{\varepsilon}(W_{\varepsilon,\xi(y)})\right)_{|_{y=0}}+o(\varepsilon)\label{eq:asexp2}
\end{eqnarray}
\end{lem}
\begin{proof}
We split the proof in several steps.

\emph{Step 1}: we prove (\ref{eq:asexp1}). Using (\ref{eq:red1})
we get 
\begin{align*}
\tilde{J}_{\varepsilon}(\xi)-J_{\varepsilon}(W_{\varepsilon,\xi})= & \frac{1}{2}\|\phi_{\varepsilon,\xi}\|_{\varepsilon}^{2}+\frac{1}{\varepsilon^{n}}\int_{M}\varepsilon^{2}g(\nabla W_{\varepsilon,\xi},\nabla \phi_{\varepsilon,\xi})+W_{\varepsilon,\xi}\phi_{\varepsilon,\xi}-f\left(W_{\varepsilon,\xi}\right)\phi_{\varepsilon,\xi}d\mu_{g}\\
 & -\frac{1}{\varepsilon^{n}}\int_{M}F\left(W_{\varepsilon,\xi}+\phi_{\varepsilon,\xi}\right)-F\left(W_{\varepsilon,\xi}\right)-f\left(W_{\varepsilon,\xi}\right)\phi_{\varepsilon,\xi}\\
=- & \frac{1}{2}\|\phi_{\varepsilon,\xi}\|_{\varepsilon}^{2}+\frac{1}{\varepsilon^{n}}\int_{M}\left[f\left(W_{\varepsilon,\xi}+\phi_{\varepsilon,\xi}\right)-f\left(W_{\varepsilon,\xi}\right)\right]\phi_{\varepsilon,\xi}d\mu_{g}\\
 & -\frac{1}{\varepsilon^{n}}\int_{M}F\left(W_{\varepsilon,\xi}+\phi_{\varepsilon,\xi}\right)-F\left(W_{\varepsilon,\xi}\right)-f\left(W_{\varepsilon,\xi}\right)\phi_{\varepsilon,\xi}
\end{align*}
By the mean value theorem we obtain that 
\[
\left|\tilde{J}_{\varepsilon}(\xi)-J_{\varepsilon}(W_{\varepsilon,\xi})\right|\le\frac{1}{2}\|\phi_{\varepsilon,\xi}\|_{\varepsilon}^{2}+\left|\frac{1}{\varepsilon^{n}}\int_{M}f'\left(W_{\varepsilon,\xi}+t_{1}\phi_{\varepsilon,\xi}\right)\phi_{\varepsilon,\xi}^{2}\right|+\left|\frac{1}{\varepsilon^{n}}\int_{M}f'\left(W_{\varepsilon,\xi}+t_{2}\phi_{\varepsilon,\xi}\right)\phi_{\varepsilon,\xi}^{2}\right|
\]
for some $t_{1},t_{2}\in(0,1)$. Now, by the properties of $f'$ we
can conclude that 
\[
\left|\tilde{J}_{\varepsilon}(\xi)-J_{\varepsilon}(W_{\varepsilon,\xi})\right|\le c\left(\|\phi_{\varepsilon,\xi}\|_{\varepsilon}^{2}+\|\phi_{\varepsilon,\xi}\|_{\varepsilon}^{p}\right)
\]
 and in light of Proposition \ref{prop:phieps} we obtain (\ref{eq:asexp1}).

\emph{Step 2: }in order to prove (\ref{eq:asexp2}), consider that
\begin{multline*}
\frac{\partial}{\partial y_{h}}J_{\varepsilon}(W_{\varepsilon,\xi(y)}+\phi_{\varepsilon,\xi(y)})-\frac{\partial}{\partial y_{h}}J_{\varepsilon}(W_{\varepsilon,\xi(y)})\\
=J_{\varepsilon}'(W_{\varepsilon,\xi(y)}+\phi_{\varepsilon,\xi(y)})\left[\frac{\partial}{\partial y_{h}}W_{\varepsilon,\xi(y)}+\frac{\partial}{\partial y_{h}}\phi_{\varepsilon,\xi(y)}\right]-J_{\varepsilon}'(W_{\varepsilon,\xi(y)})\left[\frac{\partial}{\partial y_{h}}W_{\varepsilon,\xi(y)}\right]\\
=\left[J_{\varepsilon}'(W_{\varepsilon,\xi(y)}+\phi_{\varepsilon,\xi(y)})-J_{\varepsilon}'(W_{\varepsilon,\xi(y)})\right]\left[\frac{\partial}{\partial y_{h}}W_{\varepsilon,\xi(y)}\right]\\
+J_{\varepsilon}'(W_{\varepsilon,\xi(y)}+\phi_{\varepsilon,\xi(y)})\left[\frac{\partial}{\partial y_{h}}\phi_{\varepsilon,\xi(y)}\right]=L_{1}+L_{2}.
\end{multline*}

\emph{Step 3:} we estimate $L_{2}$. We have, by (\ref{eq:red1}),
that 
\[
J_{\varepsilon}'(W_{\varepsilon,\xi(y)}+\phi_{\varepsilon,\xi(y)})\left[\frac{\partial}{\partial y_{h}}\phi_{\varepsilon,\xi(y)}\right]=\sum_{l=1}^{n-1}c_{\varepsilon}^{l}\left\langle Z_{\varepsilon,\xi(y)}^{l},\frac{\partial}{\partial y_{h}}\phi_{\varepsilon,\xi(y)}\right\rangle _{\varepsilon}.
\]
We prove that 
\begin{equation}
\sum_{l=1}^{n-1}|c_{\varepsilon}^{l}|=O(\varepsilon).\label{eq:c1}
\end{equation}
Indeed we have, by (\ref{eq:red1}) and (\ref{eq:ZiZj}), for some
positive constant $C$, 
\begin{equation}
J_{\varepsilon}'(W_{\varepsilon,\xi(y)}+\phi_{\varepsilon,\xi(y)})\left[Z_{\varepsilon,\xi(y)}^{s}\right]=\sum_{l=1}^{n-1}c_{\varepsilon}^{l}\left\langle Z_{\varepsilon,\xi(y)}^{l},Z_{\varepsilon,\xi(y)}^{s}\right\rangle _{\varepsilon}=C\sum_{l=1}^{n-1}c_{\varepsilon}^{l}(\delta_{ls}+o(1)).\label{eq:c2}
\end{equation}
Also, since $\phi_{\varepsilon,\xi(y)}\in K_{\varepsilon,\xi(y)}^{\bot}$,
we have
\begin{multline*}
J_{\varepsilon}'(W_{\varepsilon,\xi(y)}+\phi_{\varepsilon,\xi(y)})\left[Z_{\varepsilon,\xi(y)}^{s}\right]\\
=\frac{1}{\varepsilon^{n}}\int_{M}\varepsilon^{2}g(\nabla W_{\varepsilon,\xi(y)},\nabla Z_{\varepsilon,\xi(y)}^{s})+W_{\varepsilon,\xi(y)}Z_{\varepsilon,\xi(y)}^{s}-f(W_{\varepsilon,\xi(y)})Z_{\varepsilon,\xi(y)}^{s}d\mu_{g}\\
-\frac{1}{\varepsilon^{n}}\int_{M}\left[f(W_{\varepsilon,\xi(y)}+\phi_{\varepsilon,\xi(y)})-f(W_{\varepsilon,\xi(y)})\right]Z_{\varepsilon,\xi(y)}^{s}d\mu_{g}.
\end{multline*}
By (\ref{eq:g1}), (\ref{eq:g2}) and (\ref{eq:g3}), after a change
of variables we have 
\begin{multline*}
\frac{1}{\varepsilon^{n}}\int_{M}\varepsilon^{2}g(\nabla W_{\varepsilon,\xi(y)},\nabla Z_{\varepsilon,\xi(y)}^{s})+W_{\varepsilon,\xi(y)}Z_{\varepsilon,\xi(y)}^{s}-f(W_{\varepsilon,\xi(y)})Z_{\varepsilon,\xi(y)}^{s}d\mu_{g}\\
=\int_{\mathbb{R}_{+}^{n}}\nabla U\nabla\varphi^{l}+U\varphi^{l}-f(U)\varphi^{l}dz+O(\varepsilon)=O(\varepsilon).
\end{multline*}
Besides, by the mean value theorem, for some $t\in(0,1)$, 
\begin{multline*}
\left|\frac{1}{\varepsilon^{n}}\int_{M}\left[f(W_{\varepsilon,\xi(y)}+\phi_{\varepsilon,\xi(y)})-f(W_{\varepsilon,\xi(y)})\right]Z_{\varepsilon,\xi(y)}^{s}d\mu_{g}\right|\\
=\left|\frac{1}{\varepsilon^{n}}\int_{M}\left[f'(W_{\varepsilon,\xi(y)}+t\phi_{\varepsilon,\xi(y)})\right]Z_{\varepsilon,\xi(y)}^{s}\phi_{\varepsilon,\xi(y)}d\mu_{g}\right|\\
\le c\frac{1}{\varepsilon^{n}}\int_{M}\left(\left|W_{\varepsilon,\xi(y)}\right|^{p-2}+\left|\phi_{\varepsilon,\xi(y)}\right|^{p-2}\right)\left|Z_{\varepsilon,\xi(y)}^{s}\right|\left|\phi_{\varepsilon,\xi(y)}\right|d\mu_{g}\\
\le c\left(\|W_{\varepsilon,\xi(y)}\|_{\varepsilon}^{p-2}+\|\phi_{\varepsilon,\xi(y)}\|_{\varepsilon}^{p-2}\right)\|Z_{\varepsilon,\xi(y)}^{s}\|_{\varepsilon}\|\phi_{\varepsilon,\xi(y)}\|_{\varepsilon}=O(\varepsilon^{1+\frac{n}{p'}})=o(\varepsilon).
\end{multline*}
Hence $J_{\varepsilon}'(W_{\varepsilon,\xi(y)}+\phi_{\varepsilon,\xi(y)})\left[Z_{\varepsilon,\xi(y)}^{s}\right]=O(\varepsilon)$
and, comparing with (\ref{eq:c2}), we get (\ref{eq:c1}). At this
point we have, by (\ref{eq:lemma61}), (\ref{eq:c1}) and by Proposition
(\ref{prop:phieps}), that 
\begin{align*}
|L_{2}| & \le\left|\sum_{l=1}^{n-1}c_{\varepsilon}^{l}\left\langle Z_{\varepsilon,\xi(y)}^{l},\frac{\partial}{\partial y_{h}}\phi_{\varepsilon,\xi(y)}\right\rangle _{\varepsilon}\right|=\left|\sum_{l=1}^{n-1}c_{\varepsilon}^{l}\left\langle \frac{\partial}{\partial y_{h}}Z_{\varepsilon,\xi(y)}^{l},\phi_{\varepsilon,\xi(y)}\right\rangle _{\varepsilon}\right|\\
 & \le\left(\sum_{l=1}^{n-1}|c_{\varepsilon}^{l}|\right)\left\Vert \frac{\partial}{\partial y_{h}}Z_{\varepsilon,\xi(y)}^{l}\right\Vert _{\varepsilon}\left\Vert \phi_{\varepsilon,\xi(y)}\right\Vert _{\varepsilon}\le O(\varepsilon^{1+\frac{n}{p'}})=o(\varepsilon).
\end{align*}

\emph{Step 4: }we estimate $L_{1}$. We have
\begin{align*}
L_{1}= & \left\langle \phi_{\varepsilon,\xi(y)},\frac{\partial}{\partial y_{h}}W_{\varepsilon,\xi(y)}\right\rangle _{\varepsilon}-\frac{1}{\varepsilon^{n}}\int_{M}\left[f(W_{\varepsilon,\xi(y)}+\phi_{\varepsilon,\xi(y)})-f(W_{\varepsilon,\xi(y)})\right]\frac{\partial}{\partial y_{h}}W_{\varepsilon,\xi(y)}d\mu_{g}\\
= & \left\langle \phi_{\varepsilon,\xi(y)}-i_{\varepsilon}^{*}\left[f'(W_{\varepsilon,\xi(y)})\phi_{\varepsilon,\xi(y)}\right],\frac{\partial}{\partial y_{h}}W_{\varepsilon,\xi(y)}\right\rangle _{\varepsilon}\\
 & -\frac{1}{\varepsilon^{n}}\int_{M}\left[f(W_{\varepsilon,\xi(y)}+\phi_{\varepsilon,\xi(y)})-f(W_{\varepsilon,\xi(y)})-f'(W_{\varepsilon,\xi(y)})\phi_{\varepsilon,\xi(y)}\right]\frac{\partial}{\partial y_{h}}W_{\varepsilon,\xi(y)}d\mu_{g}\\
= & \left\langle \phi_{\varepsilon,\xi(y)}-i_{\varepsilon}^{*}\left[f'(W_{\varepsilon,\xi(y)})\phi_{\varepsilon,\xi(y)}\right],\frac{\partial}{\partial y_{h}}W_{\varepsilon,\xi(y)}+\frac{1}{\varepsilon}Z_{\varepsilon,\xi(y)}^{h}\right\rangle _{\varepsilon}\\
 & -\frac{1}{\varepsilon}\left\langle \phi_{\varepsilon,\xi(y)},Z_{\varepsilon,\xi(y)}^{h}-i_{\varepsilon}^{*}\left[f'(W_{\varepsilon,\xi(y)})\phi_{\varepsilon,\xi(y)}\right]\right\rangle _{\varepsilon}\\
 & -\frac{1}{\varepsilon^{n}}\int_{M}\left[f(W_{\varepsilon,\xi(y)}+\phi_{\varepsilon,\xi(y)})-f(W_{\varepsilon,\xi(y)})-f'(W_{\varepsilon,\xi(y)})\phi_{\varepsilon,\xi(y)}\right]\frac{\partial}{\partial y_{h}}W_{\varepsilon,\xi(y)}d\mu_{g}\\
= & A_{1}+A_{2}+A_{3}
\end{align*}
For the first term we have, by (\ref{eq:lemma63}) 
\begin{align*}
A_{1} & \le\|\phi_{\varepsilon,\xi(y)}-i_{\varepsilon}^{*}\left[f'(W_{\varepsilon,\xi(y)})\phi_{\varepsilon,\xi(y)}\right]\|_{\varepsilon}\left\Vert \frac{\partial}{\partial y_{h}}W_{\varepsilon,\xi(y)}+\frac{1}{\varepsilon}Z_{\varepsilon,\xi(y)}^{h}\right\Vert _{\varepsilon}\\
 & \le c\varepsilon\|\phi_{\varepsilon,\xi(y)}-i_{\varepsilon}^{*}\left[f'(W_{\varepsilon,\xi(y)})\phi_{\varepsilon,\xi(y)}\right]\|_{\varepsilon}\\
 & \le c\varepsilon\left(\|\phi_{\varepsilon,\xi(y)}\|_{\varepsilon}+-|f'(W_{\varepsilon,\xi(y)})\phi_{\varepsilon,\xi(y)}|_{p',\varepsilon}\right)\le\varepsilon\|\phi_{\varepsilon,\xi(y)}\|_{\varepsilon}=o(\varepsilon).
\end{align*}
For the second term, in light of Proposition \ref{prop:phieps} and
Equation (\ref{eq:Ziistella}), we have
\[
A_{2}\le\frac{1}{\varepsilon}\|\phi_{\varepsilon,\xi(y)}\|_{\varepsilon}\|Z_{\varepsilon,\xi(y)}^{h}-i_{\varepsilon}^{*}\left[f'(W_{\varepsilon,\xi(y)})Z_{\varepsilon,\xi(y)}^{h}\right]\|_{\varepsilon}=O(\varepsilon^{1+2\frac{n}{p'}})=o(\varepsilon).
\]
In order to estimate the last term, we have to consider separately
case $2\le p<3$ and $p\ge3$.

We recall (\cite[Remark 3.4]{MP09}) that 
\[
\left|f'(W_{\varepsilon,\xi}+v)-f'(W_{\varepsilon,\xi})\right|\le\left\{ \begin{array}{ccc}
c(p)|v|^{p-2} &  & 2<p<3\\
\\
c(p)[W_{\varepsilon,\xi}^{p-3}|v|+|v|^{p-2}] &  & p\ge3
\end{array}\right.
\]
 For $p\ge3,$ we have, by the growth properties of $f$ , and using
(\ref{eq:lemma61}), we get 
\begin{align*}
A_{3} & \le\frac{c}{\varepsilon^{n}}\int_{M}\left[|W_{\varepsilon,\xi(y)}|^{p-3}\phi_{\varepsilon,\xi(y)}^{2}+|\phi_{\varepsilon,\xi(y)}|^{p-1}\right]\left|\frac{\partial}{\partial y_{h}}W_{\varepsilon,\xi(y)}\right|d\mu_{g}\\
 & \le\|\phi_{\varepsilon,\xi(y)}\|_{\varepsilon}^{2}\left\Vert \frac{\partial}{\partial y_{h}}W_{\varepsilon,\xi(y)}\right\Vert _{\varepsilon}+\|\phi_{\varepsilon,\xi(y)}\|_{\varepsilon}^{p-1}\left\Vert \frac{\partial}{\partial y_{h}}W_{\varepsilon,\xi(y)}\right\Vert _{\varepsilon}\\
 & \le O(\varepsilon^{1+2\frac{n}{p'}})+O(\varepsilon^{p-2+(p-1)\frac{n}{p'}})=o(\varepsilon)
\end{align*}
since $p\ge3.$ 

For $2<p<3$, in a similar way, we get
\[
A_{3}\le\frac{c}{\varepsilon^{n}}\int_{M}\phi_{\varepsilon,\xi(y)}^{p-1}\left|\frac{\partial}{\partial y_{h}}W_{\varepsilon,\xi(y)}\right|d\mu_{g}\le\|\phi_{\varepsilon,\xi(y)}\|_{\varepsilon}^{p-1}\left\Vert \frac{\partial}{\partial y_{h}}W_{\varepsilon,\xi(y)}\right\Vert _{\varepsilon}=o(\varepsilon)
\]
which concludes the proof.\end{proof}
\begin{lem}
\label{lem:espJ2}It holds 
\[
J_{\varepsilon}(W_{\varepsilon,\xi})=C-\varepsilon\alpha H(\xi)+o(\varepsilon)
\]
$C^0$-uniformly with respect to $\xi$ as $\varepsilon$ goes to zero, where
\begin{eqnarray*}
C & := & \int_{\mathbb{R}_{+}^{n}}\frac{1}{2}|\nabla U(z)|^{2}+\frac{1}{2}U^{2}(z)-\frac{1}{p}U^{p}(z)dz\\
\alpha & := & \frac{\left(n-1\right)}{2}\int_{\mathbb{R}_{+}^{n}}\left(\frac{U'(|z|)}{|z|}\right)^{2}z_{n}^{3}dz
\end{eqnarray*}
\end{lem}
\begin{proof}
By definition of $J_{\varepsilon}$ we have
\begin{align*}
J_{\varepsilon}(W_{\varepsilon,\xi})= & \frac{1}{2}\int_{D^{+}(R/\varepsilon)}\sum_{i,j=1}^{n}g^{ij}(\varepsilon z)\frac{\partial\left(U(z)\chi_{R/\varepsilon}(z)\right)}{\partial z_{i}}\frac{\partial\left(U(z)\chi_{R/\varepsilon}(z)\right)}{\partial z_{j}}|g(\varepsilon z)|^{\frac{1}{2}}dz\\
 & +\int_{D^{+}(R/\varepsilon)}\left[\frac{1}{2}\left(U(z)\chi_{R/\varepsilon}(z)\right)^{2}-\frac{1}{p}\left(U(z)\chi_{R/\varepsilon}(z)\right)^{p}\right]|g(\varepsilon z)|^{\frac{1}{2}}dz.
\end{align*}
We easily get, by (\ref{eq:g1}), (\ref{eq:g2}) and (\ref{eq:g3})
\begin{eqnarray*}
J_{\varepsilon}(W_{\varepsilon,\xi}) & = & \int_{D^{+}(R/\varepsilon)}\frac{1}{2}\sum_{i=1}^{n}\frac{\partial\left(U(z)\chi_{R/\varepsilon}(z)\right)}{\partial z_{i}}\frac{\partial\left(U(z)\chi_{R/\varepsilon}(z)\right)}{\partial z_{i}}dz\\
 &  & +\int_{D^{+}(R/\varepsilon)}\frac{1}{2}\left(U(z)\chi_{R/\varepsilon}(z)\right)^{2}-\frac{1}{p}\left(U(z)\chi_{R/\varepsilon}(z)\right)^{p}dz\\
 &  & +\varepsilon\int_{D^{+}(R/\varepsilon)}\sum_{i,j=1}^{n-1}h^{ij}(0)z_{n}\frac{\partial\left(U(z)\chi_{R/\varepsilon}(z)\right)}{\partial z_{i}}\frac{\partial\left(U(z)\chi_{R/\varepsilon}(z)\right)}{\partial z_{j}}dz\\
 &  & -\frac{n-1}{2}\varepsilon\int_{D^{+}(R/\varepsilon)}H(\xi)z_{n}\sum_{i=1}^{n}\frac{\partial\left(U(z)\chi_{R/\varepsilon}(z)\right)}{\partial z_{i}}\frac{\partial\left(U(z)\chi_{R/\varepsilon}(z)\right)}{\partial z_{i}}dz\\
 &  & -\frac{n-1}{2}\varepsilon\int_{D^{+}(R/\varepsilon)}H(\xi)z_{n}\left(U(z)\chi_{R/\varepsilon}(z)\right)^{2}dz\\
 &  & +\frac{n-1}{p}\varepsilon\int_{D^{+}(R/\varepsilon)}H(\xi)z_{n}\left(U(z)\chi_{R/\varepsilon}(z)\right)^{p}dz+o(\varepsilon)=\\
 & = & \int_{\mathbb{R}_{+}^{n}}\frac{1}{2}|\nabla U(z)|^{2}+\frac{1}{2}U^{2}(z)-\frac{1}{p}U^{p}(z)dz\\
 &  & -\varepsilon(n-1)H(\xi)\int_{\mathbb{R}_{+}^{n}}z_{n}\left(\frac{1}{2}|\nabla U(z)|^{2}+\frac{1}{2}U^{2}(z)-\frac{1}{p}U^{p}(z)\right)dz\\
 &  & +\varepsilon\sum_{i,j=1}^{n-1}h^{ij}(0)\int_{\mathbb{R}_{+}^{n}}\left(\frac{U'(|z|)}{|z|}\right)^{2}z_{i}z_{j}z_{n}dz+o(\varepsilon).
\end{eqnarray*}
Using Lemma \ref{lem:dUdzn} finally we have
\begin{eqnarray*}
J_{\varepsilon}(W_{\varepsilon,\xi}) & = & \int_{\mathbb{R}_{+}^{n}}\frac{1}{2}|\nabla U(z)|^{2}+\frac{1}{2}U^{2}(z)-\frac{1}{p}U^{p}(z)dz\\
 &  & -\varepsilon(n-1)H(\xi)\int_{\mathbb{R}_{+}^{n}}\left(\frac{U'(|z|)}{|z|}\right)^{2}z_{n}^{3}dz\\
 &  & +\varepsilon\sum_{i,j=1}^{n-1}h^{ij}(0)\int_{\mathbb{R}_{+}^{n}}\left(\frac{U'(|z|)}{|z|}\right)^{2}z_{i}z_{j}z_{n}dz+o(\varepsilon).
\end{eqnarray*}
Now, by symmetry arguments and by (\ref{eq:H}) we have that 
\begin{align*}
\sum_{i,j=1}^{n-1}h^{ij}(0)\int_{\mathbb{R}_{+}^{n}}\left(\frac{U'(|z|)}{|z|}\right)^{2}z_{i}z_{j}z_{n}dz & =\sum_{i,j=1}^{n-1}h^{ij}(0)\delta_{ij}\int_{\mathbb{R}_{+}^{n}}\left(\frac{U'(|z|)}{|z|}\right)^{2}z_{i}z_{j}z_{n}dz=\\
 & =\sum_{i=1}^{n-1}h^{ii}(0)\int_{\mathbb{R}_{+}^{n}}\left(\frac{U'(|z|)}{|z|}\right)^{2}z_{i}^{2}z_{n}dz=\\
 & =(n-1)H(\xi)\int_{\mathbb{R}_{+}^{n}}\left(\frac{U'(|z|)}{|z|}\right)^{2}z_{1}^{2}z_{n}dz,
\end{align*}
and, by simple computation in polar coordinates,
\begin{equation}
\int_{\mathbb{R}_{+}^{n}}\left(\frac{U'(|z|)}{|z|}\right)^{2}z_{1}^{2}z_{n}dz=\frac{1}{2}\int_{\mathbb{R}_{+}^{n}}\left(\frac{U'(|z|)}{|z|}\right)^{2}z_{n}^{3}dz.\label{eq:z1-zn}
\end{equation}
Concluding, we get
\begin{align*}
J_{\varepsilon}(W_{\varepsilon,\xi})= & \int_{\mathbb{R}_{+}^{n}}\frac{1}{2}|\nabla U(z)|^{2}+\frac{1}{2}U^{2}(z)-\frac{1}{p}U^{p}(z)dz\\
 & -\varepsilon H(\xi)\left[\frac{(n-1)}{2}\int_{\mathbb{R}_{+}^{n}}\left(\frac{U'(|z|)}{|z|}\right)^{2}z_{n}^{3}dz\right]+o(\varepsilon),
\end{align*}
and we have the proof.\end{proof}
\begin{lem}
Let $\xi(y)=\exp_{\xi}^{\partial}(y)$, $y\in B^{n-1}(0,r)$ it holds
\[
\left.\frac{\partial}{\partial y_{h}}J_{\varepsilon}(W_{\varepsilon,\xi(y)})\right|_{y=0}=-\varepsilon\alpha\left.\left(\frac{\partial}{\partial y_{h}}H(\xi(y))\right)\right|_{y=0}+o(\varepsilon)
\]
uniformly with respect to $\xi$ as $\varepsilon$ goes to zero.\end{lem}
\begin{proof}
For simplicity, we prove the claim for $h=1$. Cases $h=2,\dots,n-1$
are straightforward. Let us consider first
\begin{eqnarray*}
\frac{\partial}{\partial y_{1}}\left.\int_{I_{\xi}(R)}\frac{1}{2\varepsilon^{n}}W_{\varepsilon,\xi(y)}^{2}d\mu_{g}\right|_{y=0} & = & \int_{I_{\xi}(R)}\frac{1}{\varepsilon^{n}}W_{\varepsilon,\xi}\frac{\partial}{\partial y_{1}}\left.W_{\varepsilon,\xi(y)}\right|_{y=0}d\mu_{g}
\end{eqnarray*}
by Lemma \ref{lem:derWeps} and by exponential decay of $U$ we have
\begin{multline*}
\frac{\partial}{\partial y_{1}}\left.\int\frac{1}{2\varepsilon^{n}}W_{\varepsilon,\xi(y)}^{2}d\mu_{g}\right|_{y=0}\\
=\int_{\mathbb{R}_{+}^{n}}U(z)\chi_{R}(\varepsilon z)\left[U(z)\frac{\partial\chi_{R}}{\partial z_{k}}(\varepsilon z)+\frac{1}{\varepsilon}\frac{\partial U}{\partial z_{k}}(z)\chi_{R}(\varepsilon z)\right]\frac{\partial}{\partial y_{1}}\left.\tilde{\mathcal{E}}(y,\varepsilon z)\right|_{y=0}|g(\varepsilon z)|^{1/2}dz\\
=\frac{1}{\varepsilon}\int_{\mathbb{R}_{+}^{n}}U(z)\frac{\partial U}{\partial z_{k}}(z)\frac{\partial}{\partial y_{1}}\left.\tilde{\mathcal{E}}(y,\varepsilon z)\right|_{y=0}|g(\varepsilon z)|^{1/2}dz+o(\varepsilon)\\
=\frac{1}{\varepsilon}\int_{\mathbb{R}_{+}^{n}}U(z)\frac{U'(z)}{|z|}z_{k}\frac{\partial}{\partial y_{1}}\left.\tilde{\mathcal{E}_{k}}(y,\varepsilon z)\right|_{y=0}|g(\varepsilon z)|^{1/2}dz+o(\varepsilon)
\end{multline*}
where $\tilde{\mathcal{E}}$ is defined in Definition \ref{def:Estorto}.
Expanding in $\varepsilon$, by Lemma \ref{lem:Estorto} and by (\ref{eq:g3})
we obtain
\begin{multline*}
\frac{\partial}{\partial y_{1}}\left.\int\frac{1}{2\varepsilon^{n}}W_{\varepsilon,\xi(y)}^{2}d\mu_{g}\right|_{y=0}\\
=\frac{1}{\varepsilon}\int_{\mathbb{R}_{+}^{n}}U(z)\frac{U'(z)}{|z|}z_{k}(-\delta_{1k}+\frac{1}{2}\varepsilon^{2}E_{ij}^{k}z_{i}z_{j})(1-\varepsilon(n-1)Hz_{n}+\frac{1}{2}\varepsilon^{2}G_{ls}z_{l}z_{s})dz+o(\varepsilon)\\
=\frac{1}{\varepsilon}\int_{\mathbb{R}_{+}^{n}}U(z)\frac{U'(z)}{|z|}\left[-z_{1}(1-\varepsilon(n-1)Hz_{n}+\frac{1}{2}\varepsilon^{2}G_{ls}z_{l}z_{s})+\frac{1}{2}z_{k}\varepsilon^{2}E_{ij}^{k}z_{i}z_{j}\right]dz+o(\varepsilon)
\end{multline*}
where $E_{ij}^{k}=\frac{\partial^{2}}{\partial z_{i}\partial z_{j}}\frac{\partial}{\partial y_{1}}\left.\tilde{\mathcal{E}_{k}}(y,z)\right|_{y=0,z=0}$
and $G_{ls}=\left.\frac{\partial^{2}}{\partial z_{i}\partial z_{j}}|g(z)|^{1/2}\right|_{z=0}$.
By symmetry reason the only terms remaining are the ones containing
$z_{r}^{2}z_{n}$, thus
\[
\frac{\partial}{\partial y_{1}}\left.\int\frac{1}{2\varepsilon^{n}}W_{\varepsilon,\xi(y)}^{2}d\mu_{g}\right|_{y=0}=\varepsilon\int_{\mathbb{R}_{+}^{n}}U(z)\frac{U'(z)}{|z|}\left[-z_{1}G_{n1}z_{n}z_{1}+z_{k}E_{kn}^{k}z_{k}z_{n}\right]dz+o(\varepsilon).
\]
By (\ref{eq:G}) we have that $G_{n1}(0)=-(n-1)\frac{\partial H}{\partial z_{1}}(0)$
and in light of Lemma \ref{lem:Estorto2} we get $E_{kn}^{k}=0.$
We conclude that
\begin{align*}
\frac{\partial}{\partial y_{1}}\left.\int\frac{1}{2\varepsilon^{n}}W_{\varepsilon,\xi(y)}^{2}d\mu_{g}\right|_{y=0} & =\varepsilon\int_{\mathbb{R}_{+}^{n}}U(z)\frac{U'(z)}{|z|}(n-1)\left(\frac{\partial H}{\partial z_{1}}(0)\right)z_{n}z_{1}^{2}dz+o(\varepsilon)\\
 & =\varepsilon\int_{\mathbb{R}_{+}^{n}}\frac{\partial}{\partial z_{1}}\left(\frac{1}{2}U^{2}(z)\right)(n-1)\left(\frac{\partial H}{\partial z_{1}}(0)\right)z_{n}z_{1}dz+o(\varepsilon)\\
 & =-\varepsilon(n-1)\frac{\partial H}{\partial z_{1}}(0)\int_{\mathbb{R}_{+}^{n}}\frac{1}{2}U^{2}(z)z_{n}dz+o(\varepsilon)
\end{align*}
 In the same way we get that
\begin{align*}
\frac{\partial}{\partial y_{1}}\left.\int\frac{1}{p\varepsilon^{n}}W_{\varepsilon,\xi(y)}^{p}d\mu_{g}\right|_{y=0} &
 =\varepsilon\int_{\mathbb{R}_{+}^{n}}U^{p-1}(z)\frac{U'(z)}{|z|}(n-1)
 \left(\frac{\partial}{\partial z_{1}}H(0)\right)z_{n}z_{1}^{2}dz+o(\varepsilon)\\
 & =-\varepsilon(n-1)\frac{\partial H}{\partial z_{1}}(0)\int_{\mathbb{R}_{+}^{n}}\frac{1}{p}U^{p}(z)z_{n}dz+o(\varepsilon)
\end{align*}
Now we look at the last term
\[
I:=\frac{\partial}{\partial y_{1}}\left.\int_{I_{\xi}(R)}\frac{\varepsilon^{2}}{2\varepsilon^{n}}
\left|\nabla W_{\varepsilon,\xi(y)}\right|^{2}_{g}d\mu_{g}\right|_{y=0}=
\int_{I_{\xi}(R)}\frac{\varepsilon^{2}}{\varepsilon^{n}}
g\left(\nabla W_{\varepsilon,\xi}\nabla \left.\frac{\partial}{\partial y_{1}}W_{\varepsilon,\xi(y)}\right)\right|_{y=0}d\mu_{g}
\]
and again, using Lemma \ref{lem:derWeps} and the decay of $U$ we
have
\[
I=\frac{1}{\varepsilon}\int_{\mathbb{R}_{+}^{n}}g^{ij}(\varepsilon z)\frac{\partial U}{\partial z_{i}}\frac{\partial}{\partial z_{j}}\left(\frac{\partial U}{\partial z_{k}}\frac{\partial}{\partial y_{1}}\left.\tilde{\mathcal{E}_{k}}(y,\varepsilon z))\right|_{y=0}\right)|g(\varepsilon z)|^{1/2}dz+o(\varepsilon)
\]
Recalling (\ref{eq:g1}) (\ref{eq:g2}) and (\ref{eq:g3}), and set,
with abuse of language, $h_{in}=h_{nj}=0$ for all $i,j=1,\dots,n$
we have
\begin{eqnarray*}
I & = & \frac{1}{\varepsilon}\int_{\mathbb{R}_{+}^{n}}(\delta_{ij}+2\varepsilon h_{ij}z_{n}+\frac{1}{2}\varepsilon^{2}\gamma_{rt}^{ij}z_{r}z_{t})(1-\varepsilon(n-1)Hz_{n}+\frac{1}{2}\varepsilon^{2}G_{ls}z_{l}z_{s})\times\\
 &  & \times\frac{\partial U}{\partial z_{i}}\frac{\partial}{\partial z_{j}}\left(\frac{\partial U}{\partial z_{k}}(-\delta_{1k}+\frac{1}{2}\varepsilon^{2}E_{vw}^{k}z_{v}z_{w})\right)dz+o(\varepsilon)
\end{eqnarray*}
where $E_{vw}^{k}=\frac{\partial^{2}}{\partial z_{v}\partial z_{w}}\frac{\partial}{\partial y_{1}}\left.\tilde{\mathcal{E}_{k}}(y,z)\right|_{y=0,z=0}$
, $G_{ls}=\left.\frac{\partial^{2}}{\partial z_{l}\partial z_{s}}|g(z)|^{1/2}\right|_{z=0}$
and $\gamma_{rt}^{ij}=\left.\frac{\partial^{2}}{\partial z_{r}\partial z_{t}}g^{ij}(z)\right|_{z=0}$.

More explicitly
\begin{align*}
I= & -\frac{1}{\varepsilon}\int_{\mathbb{R}_{+}^{n}}\frac{\partial U}{\partial z_{i}}\frac{\partial}{\partial z_{i}}\frac{\partial U}{\partial z_{1}}dz+\int_{\mathbb{R}_{+}^{n}}(n-1)Hz_{n}\frac{\partial U}{\partial z_{i}}\frac{\partial}{\partial z_{i}}\frac{\partial U}{\partial z_{1}}dz\\
 & -\frac{1}{2}\varepsilon\int_{\mathbb{R}_{+}^{n}}G_{ls}z_{l}z_{s}\frac{\partial U}{\partial z_{i}}\frac{\partial}{\partial z_{i}}\frac{\partial U}{\partial z_{1}}dz-2\int_{\mathbb{R}_{+}^{n}}h_{ij}z_{n}\frac{\partial U}{\partial z_{i}}\frac{\partial}{\partial z_{j}}\frac{\partial U}{\partial z_{1}}dz\\
 & +2\varepsilon\int_{\mathbb{R}_{+}^{n}}(n-1)Hh_{ij}z_{n}^{2}\frac{\partial U}{\partial z_{i}}\frac{\partial}{\partial z_{j}}\frac{\partial U}{\partial z_{1}}dz-\frac{1}{2}\varepsilon\int_{\mathbb{R}_{+}^{n}}\gamma_{rt}^{ij}z_{r}z_{t}\frac{\partial U}{\partial z_{i}}\frac{\partial}{\partial z_{j}}\frac{\partial U}{\partial z_{1}}dz\\
 & +\frac{1}{2}\varepsilon\int_{\mathbb{R}_{+}^{n}}\frac{\partial U}{\partial z_{i}}\frac{\partial}{\partial z_{i}}\left(\frac{\partial U}{\partial z_{k}}E_{vw}^{k}z_{v}z_{w}\right)dz+o(\varepsilon)\\
 & :=I_{1}+I_{2}+I_{3}+I_{4}+I_{5}+I_{6}+I_{7}.
\end{align*}
Easily we have 
\[
I_{1}=-\frac{1}{\varepsilon}\int_{\mathbb{R}_{+}^{n}}\frac{\partial U}{\partial z_{i}}\frac{\partial}{\partial z_{i}}\frac{\partial U}{\partial z_{1}}dz=-\frac{1}{2\varepsilon}\int_{\mathbb{R}_{+}^{n}}\frac{\partial}{\partial z_{1}}|\nabla U|^{2}dz=0,
\]
and, in a similar way, by integration by parts
\[
I_{2}=(n-1)H(0)\frac{1}{2}\int_{\mathbb{R}_{+}^{n}}z_{n}\frac{\partial}{\partial z_{1}}|\nabla U|^{2}dz=0
\]
\[
I_{3}=-\frac{1}{4}\varepsilon\int_{\mathbb{R}_{+}^{n}}G_{ls}(0)z_{l}z_{s}\frac{\partial}{\partial z_{1}}|\nabla U|^{2}dz=\varepsilon\frac{1}{2}\int_{\mathbb{R}_{+}^{n}}G_{1s}(0)z_{s}|\nabla U|^{2}dz.
\]
Moreover, by symmetry reasons, the only non zero contribution comes
from the term containing $z_{n}$, so, by (\ref{eq:G}), 
\[
I_{3}=\varepsilon\frac{1}{2}\int_{\mathbb{R}_{+}^{n}}G_{1n}(0)z_{n}|\nabla U|^{2}dz=-\varepsilon(n-1)\frac{\partial H}{\partial z_{1}}(0)\int_{\mathbb{R}_{+}^{n}}\frac{1}{2}|\nabla U|^{2}z_{n}dz.
\]
Since $h_{ij}$ is symmetric, we have 
\[
I_{4}=-\int_{\mathbb{R}_{+}^{n}}h_{ij}(0)z_{n}\frac{\partial}{\partial z_{1}}\left(\frac{\partial U}{\partial z_{i}}\frac{\partial U}{\partial z_{j}}\right)dz=0
\]
by integration by parts and, in a similar way, we obtain also that
$I_{5}=0$. 

For $I_{7}$ it holds
\[
I_{7}=\frac{1}{2}\varepsilon\int_{\mathbb{R}_{+}^{n}}\frac{\partial U}{\partial z_{i}}\frac{\partial U}{\partial z_{i}\partial z_{k}}E_{vw}^{k}z_{v}z_{w}dz+\frac{1}{2}\varepsilon\int_{\mathbb{R}_{+}^{n}}\frac{\partial U}{\partial z_{i}}\frac{\partial U}{\partial z_{k}}\left(E_{vi}^{k}z_{v}+E_{iw}^{k}z_{w}\right)dz
\]
and, since $ $$\frac{\partial U}{\partial z_{i}}=\frac{U'(z)}{|z|}z_{i}$
\begin{equation}
\int_{\mathbb{R}_{+}^{n}}\frac{\partial U}{\partial z_{i}}\frac{\partial U}{\partial z_{k}}E_{vi}^{k}z_{v}=\int_{\mathbb{R}_{+}^{n}}\left(\frac{U'(z)}{|z|}\right)^{2}E_{vi}^{k}z_{i}z_{k}z_{v}dz.\label{eq:zero}
\end{equation}
The only non zero integral are the ones of the form $z_{r}^{2}z_{n}$
and, since $E_{rr}^{n}=0$ (Lemma \ref{lem:Estorto2}), all the terms
in (\ref{eq:zero}) are $0$. With a similar argument, since $\frac{\partial U}{\partial z_{i}\partial z_{k}}=\frac{U'(z)}{|z|}\delta_{ik}+\frac{U''(z)|z|-U'}{|z|^{3}}z_{k}z_{i}$
we can conclude that $I_{7}=0$.

Finally let us consider $I_{6}.$ Since $g^{ij}$ is symmetric, integrating
by parts we have
\[
I_{6}=-\frac{1}{4}\varepsilon\int_{\mathbb{R}_{+}^{n}}\gamma_{rt}^{ij}z_{r}z_{t}\frac{\partial}{\partial z_{1}}\left(\frac{\partial U}{\partial z_{i}}\frac{\partial U}{\partial z_{j}}\right)dz=\frac{1}{2}\varepsilon\int_{\mathbb{R}_{+}^{n}}\gamma_{1t}^{ij}\left(\frac{U'(z)}{|z|}\right)^{2}z_{i}z_{j}z_{t}dz
\]
By (\ref{eq:g2}) we have that $\gamma_{1t}^{nj}=\gamma_{1t}^{in}=0$
for all $i,j,t=1,\dots,n$. In addition, for symmetry reasons, only
the term which contains $z_{r}^{2}z_{n}$ gives a non zero contribution,
and by (\ref{eq:g1}) and (\ref{eq:z1-zn}) 
\begin{align*}
I_{6} & =\frac{1}{2}\varepsilon\sum_{i=1}^{n-1}\int_{\mathbb{R}_{+}^{n}}\gamma_{1n}^{ii}\left(\frac{U'(z)}{|z|}\right)^{2}z_{i}^{2}z_{n}dz=\varepsilon\sum_{i=1}^{n-1}\frac{\partial h^{ii}}{\partial z_{1}}(0)\int_{\mathbb{R}_{+}^{n}}\left(\frac{U'(z)}{|z|}\right)^{2}z_{i}^{2}z_{n}dz.\\
 & =\varepsilon(n-1)\frac{\partial H}{\partial z_{1}}(0)\int_{\mathbb{R}_{+}^{n}}\left(\frac{U'(z)}{|z|}\right)^{2}z_{1}^{2}z_{n}dz=\frac{1}{2}\varepsilon(n-1)\frac{\partial H}{\partial z_{1}}(0)\int_{\mathbb{R}_{+}^{n}}\left(\frac{U'(z)}{|z|}\right)^{2}z_{n}^{3}dz.
\end{align*}
Concluding we have
\begin{eqnarray*}
\left(\frac{\partial}{\partial y_{1}}J_{\varepsilon}(W_{\varepsilon,\xi(y)})\right)_{|_{y=0}} & = & -\varepsilon(n-1)\frac{\partial H}{\partial z_{1}}(0)\int_{\mathbb{R}_{+}^{n}}\left[\frac{1}{2}|\nabla U|^{2}+\frac{1}{2}|U|^{2}-\frac{1}{p}|U|^{p}\right]z_{n}dz\\
 &  & +\frac{1}{2}\varepsilon(n-1)\frac{\partial H}{\partial z_{1}}(0)\int_{\mathbb{R}_{+}^{n}}\left(\frac{U'(z)}{|z|}\right)^{2}z_{n}^{3}dz.
\end{eqnarray*}
and, by Lemma \ref{lem:dUdzn}, 
\[
\left.\frac{\partial}{\partial y_{1}}J_{\varepsilon}(W_{\varepsilon,\xi(y)})\right|_{y=0}=-\frac{1}{2}\varepsilon(n-1)\frac{\partial H}{\partial z_{1}}(0)\int_{\mathbb{R}_{+}^{n}}\left(\frac{U'(z)}{|z|}\right)^{2}z_{n}^{3}dz
\]
which completes the proof.
\end{proof}

\appendix
\section{Technical lemmas}\label{app}

Here we collect a series of estimates that we used in the paper as well as the proof of some Lemma which was 
previously claimed.
\begin{lem}
\label{lem:dUdzn}It holds
\begin{multline*}
\int_{\mathbb{R}_{+}^{n}}\left(\partial_{z_{n}}U\right)^{2}z_{n}dz=\int_{\mathbb{R}_{+}^{n}}\left(\frac{U'(z)}{|z|}\right)^{2}z_{n}^{3}dz\\
=\frac{1}{2}\int_{\mathbb{R}_{+}^{n}}|\nabla U|^{2}z_{n}dz+\frac{1}{2}\int_{\mathbb{R}_{+}^{n}}U^{2}z_{n}dz-\frac{1}{p}\int_{\mathbb{R}_{+}^{n}}U^{p}z_{n}dz
\end{multline*}
\end{lem}
\begin{proof}
We multiply $-\Delta U$ by $z_{n}^{2}\partial_{z_{n}}U$, we integrate
over $\mathbb{R}_{+}^{n}$ and we integrate by parts, obtaining
\begin{align*}
-\int_{\mathbb{R}_{+}^{n}}\Delta Uz_{n}^{2}\partial_{z_{n}}Udz= & \int_{\mathbb{R}_{+}^{n}}\left(\Delta\partial_{z_{n}}U\right)z_{n}^{2}Udz+2\int_{\mathbb{R}_{+}^{n}}\Delta Uz_{n}Udz=\\
= & -\int_{\mathbb{R}_{+}^{n}}\sum_{i=1}^{n}\left(\partial_{z_{i}}\partial_{z_{n}}U\right)z_{n}^{2}\partial_{z_{i}}Udz-2\int_{\mathbb{R}_{+}^{n}}\sum_{i=1}^{n}\left(\partial_{z_{i}}\partial_{z_{n}}U\right)z_{n}\delta_{in}Udz\\
 & -2\int_{\mathbb{R}_{+}^{n}}\sum_{i=1}^{n}\left(\partial_{z_{i}}U\right)z_{n}\partial_{z_{i}}Udz-2\int_{\mathbb{R}_{+}^{n}}\sum_{i=1}^{n}\left(\partial_{z_{i}}U\right)\delta_{in}Udz=\\
= & -\int_{\mathbb{R}_{+}^{n}}\sum_{i=1}^{n}\left(\partial_{z_{i}}\partial_{z_{n}}U\right)z_{n}^{2}\partial_{z_{i}}Udz-2\int_{\mathbb{R}_{+}^{n}}\left(\partial_{z_{n}}^{2}U\right)z_{n}Udz\\
 & -2\int_{\mathbb{R}_{+}^{n}}|\nabla U|^{2}z_{n}dz-2\int_{\mathbb{R}_{+}^{n}}\left(\partial_{z_{n}}U\right)Udz.
\end{align*}
Now, again by integration by parts
\begin{align*}
-\int_{\mathbb{R}_{+}^{n}}\sum_{i=1}^{n}\left(\partial_{z_{i}}\partial_{z_{n}}U\right)z_{n}^{2}\partial_{z_{i}}Udz & =\int_{\mathbb{R}_{+}^{n}}\left(\partial_{z_{n}}U\right)z_{n}^{2}\Delta Udz+2\int_{\mathbb{R}_{+}^{n}}\sum_{i=1}^{n}\left(\partial_{z_{n}}U\right)\delta_{in}z_{n}\partial_{z_{i}}Udz\\
 & =\int_{\mathbb{R}_{+}^{n}}\left(\partial_{z_{n}}U\right)z_{n}^{2}\Delta Udz+2\int_{\mathbb{R}_{+}^{n}}\left(\partial_{z_{n}}U\right)^{2}z_{n}dz,
\end{align*}
thus
\begin{align*}
-\int_{\mathbb{R}_{+}^{n}}\Delta Uz_{n}^{2}\partial_{z_{n}}Udz= & \int_{\mathbb{R}_{+}^{n}}\left(\partial_{z_{n}}U\right)z_{n}^{2}\Delta Udz+2\int_{\mathbb{R}_{+}^{n}}\left(\partial_{z_{n}}U\right)^{2}z_{n}dz\\
 & -2\int_{\mathbb{R}_{+}^{n}}\left(\partial_{z_{n}}^{2}U\right)z_{n}Udz-2\int_{\mathbb{R}_{+}^{n}}|\nabla U|^{2}z_{n}dz-2\int_{\mathbb{R}_{+}^{n}}\left(\partial_{z_{n}}U\right)Udz
\end{align*}
that is 
\begin{align*}
-\int_{\mathbb{R}_{+}^{n}}\Delta Uz_{n}^{2}\partial_{z_{n}}Udz= & \int_{\mathbb{R}_{+}^{n}}\left(\partial_{z_{n}}U\right)^{2}z_{n}dz-\int_{\mathbb{R}_{+}^{n}}\left(\partial_{z_{n}}^{2}U\right)z_{n}Udz\\
 & -\int_{\mathbb{R}_{+}^{n}}|\nabla U|^{2}z_{n}dz-\int_{\mathbb{R}_{+}^{n}}\left(\partial_{z_{n}}U\right)Udz.
\end{align*}
Now
\[
0=\int_{\mathbb{R}_{+}^{n}}\partial_{z_{n}}\left(z_{n}U\partial_{z_{n}}U\right)dz=\int_{\mathbb{R}_{+}^{n}}U\partial_{z_{n}}U+z_{n}\left(\partial_{z_{n}}U\right)^{2}+z_{n}U\partial_{z_{n}}^{2}Udz,
\]
and we get
\begin{equation}
-\int_{\mathbb{R}_{+}^{n}}\Delta Uz_{n}^{2}\partial_{z_{n}}Udz=2\int_{\mathbb{R}_{+}^{n}}\left(\partial_{z_{n}}U\right)^{2}z_{n}dz-\int_{\mathbb{R}_{+}^{n}}|\nabla U|^{2}z_{n}dz.\label{eq:12.1}
\end{equation}

In a similar way we prove that 
\begin{eqnarray}
\int_{\mathbb{R}_{+}^{n}}Uz_{n}^{2}\partial_{z_{n}}Udz & = & -\int_{\mathbb{R}_{+}^{n}}U^{2}z_{n}dz,\label{eq:12.2}\\
\int_{\mathbb{R}_{+}^{n}}U^{p-1}z_{n}^{2}\partial_{z_{n}}Udz & = & -\frac{2}{p}\int_{\mathbb{R}_{+}^{n}}U^{p}z_{n}dz.\label{eq:12.3}
\end{eqnarray}
Now, multiplicating by $z_{n}^{2}\partial_{z_{n}}U$ both terms of
(\ref{eq:P-Rn}), integrating over $\mathbb{R}_{+}^{n}$ and using
(\ref{eq:12.1}), (\ref{eq:12.2}), (\ref{eq:12.3}) we finally obtain
the claim.
\end{proof}
The following lemma collects several estimates on $Z_{\varepsilon,\xi}^{j}$
.
\begin{lem}
There exists $\varepsilon_{0}>0$ and $c>0$ such that, for any $\xi_{0}\in\partial M$
and for any $\varepsilon\in(0,\varepsilon_{0})$ it holds
\begin{equation}
\|Z_{\varepsilon,\xi}^{h}-i^{*}\left[f'(W_{\varepsilon,\xi})Z_{\varepsilon,\xi}^{h}\right]\|_{\varepsilon}\le c\varepsilon^{1+\frac{N}{p'}}\label{eq:Ziistella}
\end{equation}
\begin{equation}
\left\Vert \frac{\partial}{\partial y_{h}}Z_{\varepsilon,\xi(y)}^{l}\right\Vert _{\varepsilon}=O\left(\frac{1}{\varepsilon}\right),\ \ \left\Vert \frac{\partial}{\partial y_{h}}W_{\varepsilon,\xi(y)}\right\Vert _{\varepsilon}=O\left(\frac{1}{\varepsilon}\right),\label{eq:lemma61}
\end{equation}
\begin{equation}
\left\langle Z_{\varepsilon,\xi_{0}}^{l},\left(\frac{\partial}{\partial y_{h}}W_{\varepsilon,\xi(y)}\right)_{|_{y=0}}\right\rangle =-\frac{1}{\varepsilon}c\delta_{lh}+o\left(\frac{1}{\varepsilon}\right),\label{eq:lemma62}
\end{equation}
\begin{equation}
\left\Vert \frac{1}{\varepsilon}Z_{\varepsilon,\xi_{0}}^{h}+\left(\frac{\partial}{\partial y_{h}}W_{\varepsilon,\xi(y)}\right)_{|_{y=0}}\right\Vert _{\varepsilon}\le c\varepsilon\label{eq:lemma63}
\end{equation}
for $h=1,\dots,n-1$, $l=1,\dots,n$\end{lem}
\begin{proof}
The proof of (\ref{eq:Ziistella}) is similar to Lemma \ref{lem:Reps}
and will be omitted. The other three estimates are similar to Lemma
6.1, Lemma 6.2 and Lemma 6.3 of \cite{MP09}, which we refer to for
the proof of the claim.
\end{proof}


\begin{proof}[Proof of Lemma \ref{lem:Linv}.]
By contradiction we assume that there exist sequences $\varepsilon_{k}\rightarrow0$,
$\xi_{k}\in\partial M$ with $\xi_{k}\rightarrow\xi\in\partial M$
and $\phi_{k}\in K_{\varepsilon_{k},\xi_{k}}^{\bot}$ with $\|\phi\|_{\varepsilon_{k}}=1$
such that 
\[
L_{\varepsilon_{k},\xi_{k}}(\phi_{k})=\psi_{k}\text{ with }\|\psi_{k}\|_{\varepsilon_{k}}\rightarrow0\text{ for }k\rightarrow+\infty.
\]
By definition of $L_{\varepsilon_{k},\xi_{k}}$, there exists $\zeta_{k}\in K_{\varepsilon_{k},\xi_{k}}$
such that
\begin{equation}
\phi_{k}-i_{\varepsilon_{k}}^{*}\left[f'(W_{\varepsilon_{k},\xi_{k}})\phi_{k}\right]=\psi_{k}+\zeta_{k}.\label{eq:zetagreca}
\end{equation}

We prove that $\|\zeta_{k}\|_{\varepsilon_{k}}\rightarrow0$ for $k\rightarrow+\infty$.
Let ${\displaystyle \zeta_{k}=\sum_{j=1}^{n-1}a_{j}^{k}Z_{\varepsilon_{k},\xi_{k}}^{j}}$,
$Z_{\varepsilon,\xi}^{i}$ being defined in (\ref{eq:Zi}). By (\ref{eq:zetagreca}),
using that $\phi_{k},\psi_{k}\in K_{\varepsilon_{k},\xi_{k}}^{\bot}$
we have
\begin{eqnarray}
\sum_{j=1}^{n-1}a_{j}^{k}\left\langle Z_{\varepsilon_{k},\xi_{k}}^{j},Z_{\varepsilon_{k},\xi_{k}}^{h}\right\rangle _{\varepsilon_{k}} & = & -\left\langle i_{\varepsilon_{k}}^{*}\left[f'(W_{\varepsilon_{k},\xi_{k}})\phi_{k}\right],Z_{\varepsilon_{k},\xi_{k}}^{h}\right\rangle _{\varepsilon_{k}}=\nonumber \\
 & = & -\frac{1}{\varepsilon_{k}^{n}}\int_{M}f'(W_{\varepsilon_{k},\xi_{k}})\phi_{k}Z_{\varepsilon_{k},\xi_{k}}^{h}d\mu_{g}.\label{eq:aj}
\end{eqnarray}
By elementary properties of $\varphi^{j}$ we have that
\begin{equation}
\begin{array}[t]{lll}
\left\langle Z_{\varepsilon_{k},\xi_{k}}^{j},Z_{\varepsilon_{k},\xi_{k}}^{j}\right\rangle _{\varepsilon_{k}}=C\delta_{jh}+o(1) &  & \text{for all }j,h=1,\dots,n-1\end{array}\label{eq:ZiZj}
\end{equation}
where $C$ is a positive constant.

We set 
\[
\tilde{\phi}_{k}:=\left\{ \begin{array}{cc}
\phi_{k}\left(\psi_{\xi_{k}}^{\partial}(\varepsilon_{k}z)\right)\chi_{R}(\varepsilon_{k}z) & \text{ if }z\in D^{+}(R)\\
\\
0 & \text{otherwise}
\end{array}\right.
\]
Easily we get that $\|\tilde{\phi}_{k}\|_{H^{1}(\mathbb{R}^{n})}\le c\|\phi_{k}\|_{\varepsilon_{k}}\le c$
for some positive constant $c$. Thus, there exists $\tilde{\phi}\in H^{1}(\mathbb{R}^{n})$
such that $\tilde{\phi}_{k}\rightarrow\tilde{\phi}$ weakly in $H^{1}(\mathbb{R}^{n})$
and strongly in $L_{\text{loc}}^{p}(\mathbb{R}^{n})$ for all $2\le p<2^{*}$
if $n\ge3$ or $p\ge2$ if $n=2$.

We recall that $\phi_{k}\in K_{\varepsilon_{k},\xi_{k}}^{\bot}$,
so 
\begin{multline*}
-\frac{1}{\varepsilon_{k}^{n}}\int_{M}f'(W_{\varepsilon_{k},\xi_{k}})\phi_{k},Z_{\varepsilon_{k},\xi_{k}}^{h}d\mu_{g}=\left\langle \phi_{k},Z_{\varepsilon_{k},\xi_{k}}^{h}\right\rangle _{\varepsilon}-\frac{1}{\varepsilon_{k}^{n}}\int_{M}f'(W_{\varepsilon_{k},\xi_{k}})\phi_{k},Z_{\varepsilon_{k},\xi_{k}}^{h}d\mu_{g}\\
=\frac{1}{\varepsilon_{k}^{n}}\int_{M}\left[\varepsilon_{k}^{2}\nabla\phi_{k}\nabla Z_{\varepsilon_{k},\xi_{k}}^{h}+Z_{\varepsilon_{k},\xi_{k}}^{h}\phi_{k}-f'(W_{\varepsilon_{k},\xi_{k}})\phi_{k},Z_{\varepsilon_{k},\xi_{k}}^{h}\right]d\mu_{g}\\
=\int_{D^{+}(R/\varepsilon_{k})}\left[\sum_{l,m=1}^{n}g^{lm}(\varepsilon_{k}z)\frac{\partial\tilde{\phi}_{k}}{\partial z_{l}}\frac{\partial\left(\varphi^{h}(z)\right)}{\partial z_{m}}+\tilde{\phi}_{k}\varphi^{h}(z)\right]|g(\varepsilon_{k}z)|^{\frac{1}{2}}dy\\
-\int_{D^{+}(R/\varepsilon_{k})}f'(U(z)\chi_{R}(\varepsilon_{k}z))\tilde{\phi}_{k}\varphi^{h}(z)|g(\varepsilon_{k}z)|^{\frac{1}{2}}dz+o(1)\\
=\int_{\mathbb{R}^{n}}\nabla\tilde{\phi}\nabla\varphi^{h}+\tilde{\phi}\varphi^{h}-f'(U)\tilde{\phi}\varphi^{h}dz+o(1)=o(1)
\end{multline*}
because $\varphi^{h}$ is a weak solution of the linearized problem
(\ref{eq:linear}). So we can rewrite (\ref{eq:aj}), obtaining 
\[
ca_{h}^{k}+o(1)=\sum_{j=1}^{n-1}a_{j}^{k}\left\langle Z_{\varepsilon_{k},\xi_{k}}^{j},Z_{\varepsilon_{k},\xi_{k}}^{h}\right\rangle _{\varepsilon_{k}}=-\frac{1}{\varepsilon_{k}^{n}}\int_{M}f'(W_{\varepsilon_{k},\xi_{k}})\phi_{k},Z_{\varepsilon_{k},\xi_{k}}^{h}d\mu_{g}=o(1),
\]
so $a_{h}^{k}\rightarrow0$ for all $h$ while $k\rightarrow+\infty$,
thus $\|\zeta_{k}\|_{\varepsilon_{k}}\rightarrow0$ for $k\rightarrow+\infty$.

Setting $u_{k}:=\phi_{k}-\psi_{k}-\zeta_{k}$, (\ref{eq:zetagreca})
can be read as
\begin{equation}
\left\{ \begin{array}{cc}
-\varepsilon_{k}^{2}\Delta_{g}u_{k}+u_{k}=f'(W_{\varepsilon_{k},\xi_{k}})u_{k}+f'(W_{\varepsilon_{k},\xi_{k}})(\psi_{k}+\zeta_{k}) & \text{ in }M\\
\\
{\displaystyle \frac{\partial u_{k}}{\partial\nu}=0} & \text{ on }\partial M.
\end{array}\right.\label{eq:ukappa}
\end{equation}
Multiplying (\ref{eq:ukappa}) by $u_{k}$ and integrating by parts
we get
\begin{equation}
\|u_{k}\|_{\varepsilon_{k}}=\frac{1}{\varepsilon_{k}^{n}}\int_{M}f'(W_{\varepsilon_{k},\xi_{k}})u_{k}^{2}+f'(W_{\varepsilon_{k},\xi_{k}})(\psi_{k}+\zeta_{k})u_{k}.\label{eq:u1}
\end{equation}
By Holder inequality, and recalling that $|u|_{\varepsilon,p}\le c\|u\|_{\varepsilon}$,
we have 
\begin{align}
\frac{1}{\varepsilon_{k}^{n}}\int_{M}f'(W_{\varepsilon_{k},\xi_{k}})(\psi_{k}+\zeta_{k})u_{k}\le & \left(\frac{1}{\varepsilon_{k}^{n}}\int_{M}f'(W_{\varepsilon_{k},\xi_{k}})^{\frac{n}{2}}\right)^{\frac{2}{n}}|u_{k}|_{\varepsilon_{k},\frac{2n}{n-2}}^{\frac{n-2}{2n}}|\psi_{k}+\zeta_{k}|_{\varepsilon_{k},\frac{2n}{n-2}}^{\frac{n-2}{2n}}\nonumber \\
\le & c\left(\frac{1}{\varepsilon_{k}^{n}}\int_{M}f'(W_{\varepsilon_{k},\xi_{k}})^{\frac{n}{2}}\right)^{\frac{2}{n}}\|u_{k}\|_{\varepsilon_{k}}^{\frac{n-2}{2n}}\|\psi_{k}+\zeta_{k}\|_{\varepsilon_{k}}^{\frac{n-2}{2n}}.\label{eq:u2}
\end{align}
Now, 
\begin{align}
\frac{1}{\varepsilon_{k}^{n}}\int_{M}f'(W_{\varepsilon_{k},\xi_{k}})^{\frac{n}{2}}d\mu_{g} & \le\frac{1}{\varepsilon_{k}^{n}}\int_{I_{\xi_{k}}(R)}\left(U_{\varepsilon}\left(\left(\psi_{\xi_{k}}^{\partial}\right)^{-1}(x)\right)\right)^{\frac{n(p-2)}{2}}d\mu_{g}\nonumber \\
 & \le c\int_{D^{+}(R/\varepsilon)}\left(U\left(z\right)\right)^{\frac{n(p-2)}{2}}dz\le c\label{eq:u3}
\end{align}
for some positive constant $c$. 

Combining (\ref{eq:u1}), (\ref{eq:u2}), (\ref{eq:u3}), and recalling
that $\|u_{k}\|_{\varepsilon_{k}}\rightarrow1$, $\|\psi_{k}+\zeta_{k}\|_{\varepsilon_{k}}\rightarrow0$

while $k\rightarrow+\infty,$we get
\begin{equation}
\frac{1}{\varepsilon_{k}^{n}}\int_{M}f'(W_{\varepsilon_{k},\xi_{k}})u_{k}^{2}\rightarrow1\text{ while }k\rightarrow+\infty.\label{eq:lim1}
\end{equation}
We will see how this leads us to a contradiction. 

We set 
\[
\tilde{u}_{k}(z):=u_{k}\left(\psi_{\xi_{k}}^{\partial}(\varepsilon_{k}z)\right)\chi_{R}(\varepsilon_{k}z)\text{ for }z\in\mathbb{R}_{+}^{n}
\]

We have that 
\[
\|\tilde{u}_{k}\|_{H^{1}(\mathbb{R}^{n})}\le c\|u_{k}\|_{\varepsilon_{k}}\le c,
\]
so, up to subsequence, there exists $\tilde{u}\in H^{1}(\mathbb{R}_{+}^{n})$
such that $\tilde{u}_{k}\rightarrow\tilde{u}$ weakly in $H^{1}(\mathbb{R}_{+}^{n})$and
strongly in $L_{\text{loc}}^{p}(\mathbb{R}_{+}^{n})$, $p\in(2,2^{*})$
if $n\ge3$ or $p>2$ if $n=2$. By (\ref{eq:ukappa}) we deduce that
\begin{equation}
\left\{ \begin{array}{cc}
-\Delta\tilde{u}+\tilde{u}=f'(U)\tilde{u} & \text{ in }\mathbb{R}_{+}^{n}\\
\\
{\displaystyle \frac{\partial\tilde{u}}{\partial x_{n}}=0} & \text{ on }\left\{ x_{n}=0\right\} .
\end{array}\right.\label{eq:utilde}
\end{equation}

We prove also that 
\begin{equation}
\left\langle \varphi^{h},\tilde{u}\right\rangle _{H^{1}}=0\text{ for all }h\in1,\dots,n-1.\label{eq:scalare}
\end{equation}
In fact, since $\phi_{k},\psi_{k}\in K_{\varepsilon,\xi}^{\bot}$
and $\|\zeta_{k}\|_{\varepsilon_{k}}\rightarrow0$, we have
\begin{equation}
\left|\left\langle Z_{\varepsilon_{k},\xi_{k}}^{h},u_{k}\right\rangle _{\varepsilon_{k}}\right|=\left|\left\langle Z_{\varepsilon_{k},\xi_{k}}^{h},\zeta_{k}\right\rangle _{\varepsilon_{k}}\right|\le\|Z_{\varepsilon_{k},\xi_{k}}^{h}\|_{\varepsilon_{k}}\|\zeta_{k}\|_{\varepsilon_{k}}=o(1).\label{eq:zetau}
\end{equation}
On the other hand, by direct computation, we get 
\begin{eqnarray}
\left\langle Z_{\varepsilon_{k},\xi_{k}}^{h},u_{k}\right\rangle _{\varepsilon_{k}} & = & \frac{1}{\varepsilon_{k}^{n}}\int_{M}\varepsilon_{k}^{2}g(\nabla Z_{\varepsilon_{k},\xi_{k}}^{h}\nabla u_{k})+Z_{\varepsilon_{k},\xi_{k}}^{h}u_{k}\nonumber \\
 & = & \int_{D^{+}(R/\varepsilon_{k})}\sum_{l,m=1}^{n}g^{lm}(\varepsilon_{k}z)\frac{\partial\left(\varphi^{h}(z)\chi_{R}(\varepsilon_{k}z)\right)}{\partial z_{l}}\frac{\partial\tilde{u}_{k}}{\partial z_{m}}|g(\varepsilon_{k}z)|^{\frac{1}{2}}dz\nonumber \\
 &  & +\int_{D^{+}(R/\varepsilon_{k})}\varphi^{h}(z)\chi_{R}(\varepsilon_{k}z)\tilde{u}_{k}|g(\varepsilon_{k}z)|^{\frac{1}{2}}dz\nonumber \\
 & = & \int_{\mathbb{R}_{+}^{n}}\left(\nabla\varphi^{h}\nabla\tilde{u}+\varphi^{h}\tilde{u}\right)dz+o(1).\label{eq:zetau2}
\end{eqnarray}
So, by (\ref{eq:zetau}) and (\ref{eq:zetau2}) we obtain (\ref{eq:scalare}). 

Now (\ref{eq:scalare}) and (\ref{eq:utilde}) imply that $\tilde{u}=0$.
Thus

\begin{eqnarray*}
\frac{1}{\varepsilon_{k}^{n}}\int_{M}f'(W_{\varepsilon_{k},\xi_{k}}(x))u_{k}^{2}(x)d\mu_{g} & \le & \frac{1}{\varepsilon_{k}^{n}}\int_{I_{g}(R)}f'\left(U_{\varepsilon}\left(\left(\psi_{\xi_{k}}^{\partial}\right)^{-1}(x)\right)\right)u_{k}^{2}(x)d\mu_{g}\\
 & = & c\int_{D^{+}(R/\varepsilon_{k})}f'(U(z))\tilde{u}_{k}^{2}(z)=o(1)
\end{eqnarray*}
which contradicts (\ref{eq:lim1}). This concludes the proof.
\end{proof}


\end{document}